\documentclass[12pt]{amsart}
\usepackage{graphicx}
\usepackage[all]{xy}
\newtheorem{theorem}{Theorem}[section]
\newtheorem{proposition}[theorem]{Proposition}
\newtheorem{lemma}[theorem]{Lemma}

\newtheorem{corollary}[theorem]{Corollary}

\theoremstyle{definition}
\newtheorem{definition}[theorem]{Definition}
\newtheorem{example}[theorem]{Example}

\theoremstyle{remark}
\newtheorem{remark}[theorem]{Remark}

\numberwithin{equation}{section}

\begin{document}


\title[Homology cylinders and sutured manifolds for knots]
{Homology cylinders and sutured manifolds for homologically fibered knots}
\date{\today}

\subjclass[2000]{Primary 57M27, Secondary 57M05, 57M25}

\author{Hiroshi Goda}
\address{Department of Mathematics,
Tokyo University of Agriculture and Technology,
2-24-16 Naka-cho, Koganei,
Tokyo 184-8588, Japan}
\email{goda@cc.tuat.ac.jp}
\author{Takuya Sakasai}
\address{Graduate School of Mathematical Sciences, 
The University of Tokyo, 
3-8-1 Komaba Meguro-ku Tokyo 153-8914, Japan.}
\email{sakasai@ms.u-tokyo.ac.jp}


\dedicatory{Dedicated to Professor Akio Kawauchi on the occasion 
of his 60th birthday}

\keywords{Homology cylinder, homologically fibered knot, sutured manifold,  
Magnus representation, Alexander polynomial, torsion, Nakanishi index}

\begin{abstract}
Sutured manifolds defined by 
Gabai are useful in the geometrical study of 
knots and 3-dimensional manifolds. 
On the other hand, homology cylinders are in an important 
position in the recent theory of 
homology cobordisms of surfaces and finite-type invariants. 
We study a relationship between them by focusing on 
sutured manifolds associated with a special class of knots 
which we call {\it homologically fibered knots}. 
Then we use invariants of homology cylinders 
to give applications to knot theory such as fibering obstructions, 
Reidemeister torsions and handle numbers of 
homologically fibered knots. 
\end{abstract}

\maketitle

\section{Introduction}\label{sec:intro}

In the theory of knots and 3-manifolds, sutured manifolds play an important 
role. They were defined by Gabai \cite{gabai1} and are used to construct 
taut foliations on 3-manifolds. 
To each knot in the 3-sphere $S^3$ with a Seifert surface $R$, 
a sutured manifold $(M,\gamma)$ called 
the {\it complementary sutured manifold} for $R$ is obtained 
by cutting the knot complement along $R$ 
with the resulting cobordism $M$ between two copies of $R$. 
Using taut foliations on complementary 
sutured manifolds, Gabai settled, for example, 
Property R conjecture \cite{gabai3}. 

On the other hand, {\it homology cylinders}, each of which 
consists of a homology cobordism $M$ between two copies of a compact surface 
and markings of both sides of the boundary of $M$ 
(see Section \ref{section:cylinder} for details), 
appeared in the context of the theory of finite type 
invariants for 3-manifolds. 
The set of homology cylinders over a surface 
has a natural monoid structure. 
Goussarov \cite{gou}, 
Habiro \cite{habiro}, Garoufalidis-Levine \cite{gl} and 
Levine \cite{levin} studied it systematically 
by using the clasper (or clover) surgery theory. 

Since both sutured manifolds and homology cylinders deal with 
cobordisms between surfaces, 
it is natural to observe their precise relationship. 
In this paper, we first give a specific answer 
by restricting sutured manifolds to 
those obtained from knots. That is, we 
discuss which knot and its Seifert surface 
define a homology cylinder as a complementary sutured manifold 
and conclude in Section \ref{section:fibre} that 
such a case occurs exactly when we 
take a knot with a minimal genus Seifert surface 
whose Alexander polynomial is monic and has degree 
twice the genus of the knot 
(see Proposition \ref{thm:homologicalfibre}, 
where the cases of links are also discussed).  
We call such a knot a {\it homologically fibered knot}. 
Several examples of homologically fibered knots are 
presented in the same section. 

It is well known that fibered knots satisfy the above conditions for 
homologically fibered knots. In fact, 
they define homology cylinders with the 
product cobordism on a surface with markings 
(called {\it monodromies} in the theory of fibered knots). 
On the other hand, 
interesting examples of homologically fibered knots come from 
non-fibered knots. 
They give homology cylinders whose underlying cobordisms 
are not product. 
To construct such homology cylinders, it has been known the following 
methods: 
\begin{itemize}
\item connected sums of the trivial cobordism with homology 3-spheres; 
\item Levine's method \cite[Section 3]{levin} 
using string links in the 3-ball;
\item Habegger's method \cite{habegger} giving homology cylinders as 
results of surgeries along string links in homology 3-balls; and 
\item clasper surgeries (see \cite{gou} and \cite{habiro}).
\end{itemize}
It was shown that each of the latter two methods 
(together with changes of markings) 
give all homology cylinders. However those methods need 
surgeries along links with multiple components, 
so that it seems slightly difficult to 
imagine the resulting manifolds. 
Our result in Section \ref{section:fibre} 
shall provide an {\it explicit} 
construction of homology cylinders. 

The above mentioned relationship between sutured manifolds 
and homology cylinders will be studied further 
in the latter half of this paper. 
We apply invariants of homology cylinders 
defined in \cite{sakasai} to homologically 
fibered knots. In particular, 
we focus on Magnus representations and 
Reidemeister torsions of homology cylinders, whose 
definitions are recalled in Section \ref{sec:magnus}. 
The definitions will be given in such a general form that we can 
apply frameworks of Cochran-Orr-Teichner's theory \cite{cot} of 
higher-order Alexander modules and Friedl's theory \cite{fri} 
of noncommutative Reidemeister torsions. 
As an immediate application, 
it turns out that they give fibering obstructions of homologically 
fibered knots. We also use them to derive 
factorization formulas of Reidemeister torsions of the exterior of 
a homologically fibered knot in Section \ref{sec:factorization}. 

More applications are given in Sections \ref{section:handle} 
and \ref{sec:nakanishi}. 
We consider {\it handle numbers} of 
sutured manifolds, which 
may be regarded as 
an analogue of the Heegaard genus of a closed 3-manifold 
for a sutured manifold. See \cite{goda1,goda2} for details. 
We discuss lower estimates 
of handle numbers by using the above mentioned invariants of 
homology cylinders. 
In particular, 
we consider doubled knots with certain Seifert surfaces and 
give a lower bound of their handle numbers by using 
Nakanishi index \cite{kawauchi2}. 

Conversely, an application of homologically fibered knots 
to homology cylinders is given in \cite{gs09}, where 
we discuss abelian quotients of monoids of homology cylinders.

The authors would like to thank Professor Yasutaka Nakanishi 
for his helpful comments. They also would like to thank 
Professor Gw\'ena\"el Massuyeau for his careful reading of the 
previous version of this paper and useful comments. 
The authors are partially supported
by Grant-in-Aid for Scientific Research,
(No.~18540072 and No.~19840009),
Ministry of Education, Science, 
Sports and Technology, Japan. 
The second author is also supported by 
21st century COE program at Graduate School of Mathematical
Sciences, The University of Tokyo.


\section{Homology cylinders and sutured manifolds}\label{section:cylinder}

In this section, we introduce two main objects in this paper: 
homology cylinders and sutured manifolds. 
First, we define homology cylinders over surfaces, 
which have their origin in 
following Goussarov \cite{gou}, Habiro \cite{habiro}, 
Garoufalidis-Levine \cite{gl} and Levine \cite{levin}. 
Let $\Sigma_{g,n}$ be a compact connected 
oriented surface of genus $g\ge 0$ with $n \ge 1$ 
boundary components.

\begin{definition}
A {\it homology cylinder\/}  $(M,i_{+},i_{-})$ over $\Sigma_{g,n}$ 
consists of a compact oriented 3-manifold $M$ 
with two embeddings $i_{+}, i_{-}: \Sigma_{g,n} \hookrightarrow \partial M$ 
such that:
\begin{enumerate}
\renewcommand{\labelenumi}{(\roman{enumi})}
\item
$i_{+}$ is orientation-preserving and $i_{-}$ is orientation-reversing; 
\item 
$\partial M=i_{+}(\Sigma_{g,n})\cup i_{-}(\Sigma_{g,n})$ and 
$i_{+}(\Sigma_{g,n})\cap i_{-}(\Sigma_{g,n})
=i_{+}(\partial\Sigma_{g,n})=i_{-}(\partial\Sigma_{g,n})$;
\item
$i_{+}|_{\partial \Sigma_{g,n}}=i_{-}|_{\partial \Sigma_{g,n}}$; and 
\item
$i_{+},i_{-} : H_{*}(\Sigma_{g,n})\to H_{*}(M)$ 
are isomorphisms. 
\end{enumerate}
If we replace (iv) with the condition that
$i_{+},i_{-} : H_{*}(\Sigma_{g,n};\mathbb Q)\to H_{*}(M;\mathbb Q)$ 
are isomorphisms, 
then $(M,i_{+},i_{-})$ is called a {\it rational homology cylinder\/}. 
\end{definition}
\noindent
We often write a (rational) homology cylinder $(M,i_+,i_-)$ briefly by $M$. 
Precisely speaking, our definition is the same as that 
in \cite{gl} and \cite{levin} except that 
we may consider homology cylinders over surfaces 
with multiple boundaries. 

Two (rational) homology cylinders $(M,i_+,i_-)$ and $(N,j_+,j_-)$ 
over $\Sigma_{g,n}$ are said to be {\it isomorphic} if there exists 
an orientation-preserving diffeomorphism $f:M \xrightarrow{\cong} N$ 
satisfying $j_+ = f \circ i_+$ and $j_- = f \circ i_-$. 
We denote the set of isomorphism classes of homology cylinders 
(resp.~rational homology cylinders) over 
$\Sigma_{g,n}$ by $\mathcal{C}_{g,n}$ 
(resp.~$\mathcal{C}_{g,n}^\mathbb{Q}$). 

\begin{example}\label{ex:mgtocg}
For each diffeomorphism $\varphi$ of 
$\Sigma_{g,n}$ which fixes $\partial \Sigma_{g,n}$ pointwise 
(hence, $\varphi$ preserves the orientation of $\Sigma_{g,n}$), 
we can construct a homology cylinder 
by setting 
\[(\Sigma_{g,n} \times [0,1], \mathrm{id} \times 1, 
\varphi \times 0),\]
where collars of $i_+ (\Sigma_{g,n})$ and $i_- (\Sigma_{g,n})$ 
are stretched half-way along $(\partial \Sigma_{g,n}) \times [0,1]$. 
It is easily checked that the isomorphism class of 
$(\Sigma_{g,n} \times [0,1], \mathrm{id} \times 1, \varphi \times 0)$ 
depends only on the (boundary fixing) 
isotopy class of $\varphi$. Therefore, 
this construction gives a map from the mapping class group 
$\mathcal{M}_{g,n}$ of $\Sigma_{g,n}$ to $\mathcal{C}_{g,n}$. 
\end{example}

Given two (rational) homology cylinders $M=(M,i_+,i_-)$ and 
$N=(N,j_+,j_-)$ over $\Sigma_{g,n}$, we can construct a new one 
defined by 
\[M \cdot N :=(M \cup_{i_- \circ (j_+)^{-1}} N, i_+,j_-).\]
By this operation, $\mathcal{C}_{g,n}$ and $\mathcal{C}_{g,n}^\mathbb{Q}$ 
become monoids with the unit 
$(\Sigma_{g,n} \times [0,1], \mathrm{id} \times 1, \mathrm{id} \times 0)$. 
The map $\mathcal{M}_{g,n} \to \mathcal{C}_{g,n}$ 
in Example \ref{ex:mgtocg} is seen to be a monoid homomorphism. 

By definition, we can define a homomorphism 
$\sigma:\mathcal{C}_{g,n} \to \mathrm{Aut} (H_1 (\Sigma_{g,n}))$ by 
\[\sigma(M,i_+,i_-) := i_+^{-1} \circ i_- \in 
\mathrm{Aut} (H_1 (\Sigma_{g,n})),\]
where $i_+$ and $i_-$ in the right hand side are the induced maps 
on the first homology. 
Note that the composition 
\[\mathcal{M}_{g,n} \xrightarrow{\mathrm{Example}\ \ref{ex:mgtocg}} 
\mathcal{C}_{g,n} \xrightarrow{\ \sigma\ } 
\mathrm{Aut} (H_1 (\Sigma_{g,n}))\]
is just the map obtained as the natural action of 
$\mathcal{M}_{g,n}$ on $H_1 (\Sigma_{g,n})$. 
For rational homology cylinders, 
we have a similar homomorphism 
\[\sigma :\mathcal{C}_{g,n}^\mathbb{Q} \to \mathrm{Aut} 
(H_1 (\Sigma_{g,n};\mathbb{Q})).\] 

The following facts seem to be well known at least for $n=1$ 
(see \cite[Section 2.4]{gl} and \cite[Section 2.1]{levin}). 
However, here we give a direct and topological proof of them. 

\begin{proposition}\label{prop:monodromy}
$(1)$ The homomorphism $\mathcal{M}_{g,n} \to \mathcal{C}_{g,n}$ 
in Example $\ref{ex:mgtocg}$ is injective. 

\begin{itemize}
\item[$(2)$] For each $(M,i_+,i_-) \in \mathcal{C}_{g,n}$, 
the automorphism $\sigma(M) := \sigma(M,$ $i_+, i_-)$ preserves 
the intersection pairing on $H_1 (\Sigma_{g,n})$. 
$($A similar statement obtained by 
replacing $H_1 (\Sigma_{g,n})$ with 
$H_1 (\Sigma_{g,n};\mathbb{Q})$ holds for rational homology cylinders.$)$
\end{itemize}
\end{proposition}
\begin{proof} 
(1) Suppose $[\varphi] \in \mathrm{Ker} (\mathcal{M}_{g,n} \to 
\mathcal{C}_{g,n})$. We may assume that the diffeomorphism 
$\varphi$ is the identity map near $\partial \Sigma_{g,n}$. 
By assumption, 
there exists a diffeomorphism $\Phi: \Sigma_{g,n} \times [0,1] 
\xrightarrow{\cong} \Sigma_{g,n} \times [0,1]$ satisfying 
\[\Phi \big|_{\Sigma_{g,n} \times \{1\}} = 
\mathrm{id}_{\Sigma_{g,n}} \times \{1\}, \quad 
\Phi \big|_{(\partial \Sigma_{g,n}) \times [0,1]} = 
\mathrm{id}_{(\partial \Sigma_{g,n}) \times [0,1]} \ \ \mathrm{and}\ \ 
\Phi \big|_{\Sigma_{g,n} \times \{0\}} = 
\varphi \times \{0\}.\]
Let $\varphi_t$ $(0 \le t \le 1)$ be the map defined as the composite
\[\Sigma_{g,n} \xrightarrow{\mathrm{id} \times \{t\}} 
\Sigma_{g,n} \times [0,1] 
\xrightarrow{\ \Phi\ } \Sigma_{g,n} \times [0,1] 
\xrightarrow{\mathrm{projection}} \Sigma_{g,n}.\]
Then $\{\varphi_t\}_{0 \le t \le 1}$ 
gives a homotopy between 
$\varphi_0=\mathrm{id}_{\Sigma_{g,n}}$ and $\varphi_1=\varphi$. 
It is well known (see \cite[Section 2]{ivanov} and references 
given there) that for the surface $\Sigma_{g,n}$ we are now considering, 
two diffeomorphisms connected by a boundary fixing 
homotopy are isotopic. Hence $\varphi$ is isotopic to 
the identity and so $[\varphi]=1 \in \mathcal{M}_{g,n}$. 

(2) Recall that the intersection pairing $\langle \ , \, 
\rangle_{\Sigma_{g,n}} :H_1(\Sigma_{g,n}) 
\otimes H_1(\Sigma_{g,n}) \to \mathbb{Z}$
on $H_1 (\Sigma_{g,n})$ is defined as 
the composition
\[H_1(\Sigma_{g,n}) \otimes H_1(\Sigma_{g,n}) \to 
H_1(\Sigma_{g,n}) \otimes 
H_1(\Sigma_{g,n}, \partial \Sigma_{g,n}) 
\xrightarrow{\,\cong\,} 
H_1(\Sigma_{g,n}) \otimes H^1(\Sigma_{g,n}) \to 
\mathbb{Z},\]
where the first (resp.~second) map is applying the natural map 
$H_1(\Sigma_{g,n}) \to H_1(\Sigma_{g,n}, \partial \Sigma_{g,n})$ 
(resp.~the Poincar\'e duality) to the second factor and 
the last map is the Kronecker product. 

The boundary $\partial M$ of $M$ is the double of $\Sigma_{g,n}$ 
so that it is a closed oriented 
surface of genus $2g+n-1$. It is easy to see that the intersection pairing 
$\langle \ , \, \rangle_{\partial M}$ on $H_1 (\partial M)$ 
satisfies
\[\langle x , y \rangle_{\Sigma_{g,n}} = 
\langle i_+(x) , i_+(y) \rangle_{\partial M} = 
-\langle i_-(x) , i_-(y) \rangle_{\partial M}
\]
for any $x, y \in H_1(\Sigma_{g,n})$. 
Also, the intersection pairing 
$\langle \ , \, \rangle_M : H_1 (M) \otimes H_2(M, \partial M) \to \mathbb{Z}$ 
on $M$ satisfies 
\[\langle i(x), Y \rangle_M = 
- \langle x , \partial Y \rangle_{\partial M}\]
for any $x \in H_1 (\partial M)$ and $Y \in H_2 (M, \partial M)$, 
where $i:\partial M \hookrightarrow M$ denotes the inclusion. 
Then our claim follows from 
\begin{align*}
\langle x , y \rangle_{\Sigma_{g,n}} 
&= -\langle i_-(x) , i_-(y) \rangle_{\partial M} 
= -\langle i_-(x) , i_-(y) - i_+ (\sigma(M) (y)) \rangle_{\partial M} \\
&= \langle i_- (x), Y \rangle_M 
= \langle i_+ (\sigma(M) (x)), Y \rangle_M \\
&= -\langle i_+(\sigma(M)(x)), i_-(y) - i_+ (\sigma(M) (y)) 
\rangle_{\partial M} 
= \langle i_+(\sigma(M) (x)), i_+(\sigma(M) (y)) \rangle_{\partial M} \\
&= \langle \sigma(M)(x), \sigma(M) (y) \rangle_{\Sigma_{g,n}}, 
\end{align*}
where $Y \in H_2 (M,\partial M)$ is a homology class satisfying 
$\partial Y = i_-(y) - i_+ (\sigma(M) (y))$. 
\end{proof}

To represent $\sigma(M,i_+,i_-)$ by a matrix, we here and hereafter 
fix a spine $S$ of $\Sigma_{g,n}$ as in Figure \ref{fig:spine0}. 
That is, $S$ is a bouquet of oriented $2g+n-1$ circles 
$\gamma_1, \ldots , \gamma_{2g+n-1}$ 
tied at a base point $p \in \partial \Sigma_{g,n}$ 
such that it is deformation retract of 
$\Sigma_{g,n}$ relative to $p$. 
The fundamental group $\pi_1 (\Sigma_{g,n})$ of $\Sigma_{g,n}$ is 
the free group $F_{2g+n-1}$ of rank $2g+n-1$ generated by 
$\gamma_1,\ldots,\gamma_{2g+n-1}$. 
These loops form an ordered basis of 
$H_1 (\Sigma_{g,n}) \cong \mathbb{Z}^{2g+n-1}$. 

\begin{remark}\label{rmk:symplectic}
Let $(M,i_+,i_-) \in \mathcal{C}_{g,n}$. 
Proposition \ref{prop:monodromy} (2) and its proof show that 
$\sigma(M,i_+,i_-) \in 
\mathrm{Aut}(H_1 (\Sigma_{g,n})) \cong GL(2g+n-1,\mathbb{Z})$ 
is represented by a matrix of the form
\[\left(\begin{array}{cc}
X & 0_{(2g,n-1)} \\ \ast & I_{n-1}
\end{array}\right)\]
with $X \in \mathrm{Sp}(2g,\mathbb{Z})$. (A similar result using 
$\mathrm{Sp}(2g,\mathbb{Q})$ holds for $\mathcal{C}_{g,n}^\mathbb{Q}$.) 

\begin{figure}[htbp]
\begin{center}
\includegraphics[width=0.6\textwidth]{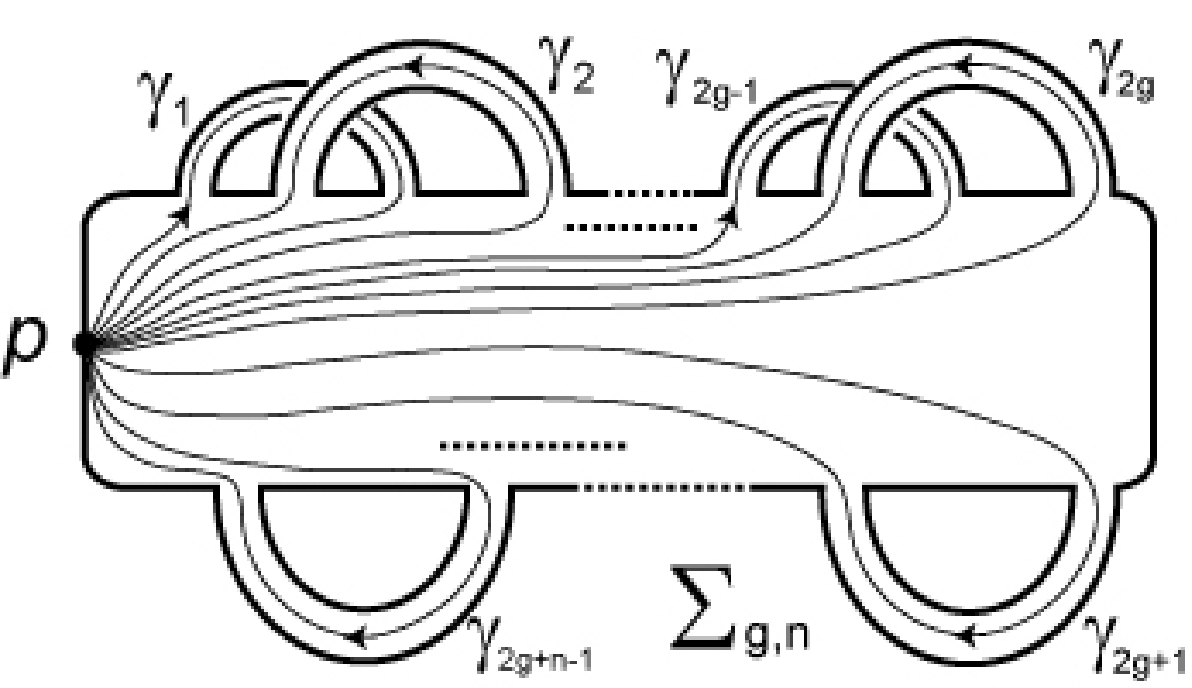}
\end{center}
\caption{A spine $S$ of $\Sigma_{g,n}$}
\label{fig:spine0}
\end{figure}

\end{remark}

Next we recall the definition of sutured manifolds given by Gabai 
\cite{gabai1}. 
We use here a special class of sutured manifolds.

\begin{definition}\label{def:suture}
A {\it sutured manifold\/} $(M,\gamma)$ is a compact 
oriented 3-manifold $M$ together with a subset $\gamma\subset \partial M$ 
which is a union of finitely many mutually disjoint annuli. 
For each component of $\gamma$, 
an oriented core circle called a {\it suture\/} is fixed, 
and we denote the set 
of sutures by $s(\gamma)$. 
Every component of $R(\gamma)=\partial M-{\rm Int\,}\gamma$
is oriented so that the orientations on $R(\gamma)$ are coherent 
with respect to $s(\gamma)$, i.e., the orientation of each component 
of $\partial R(\gamma)$, which is induced by that of $R(\gamma)$, 
is parallel to the orientation of the corresponding component 
of $s(\gamma)$. 
We denote by $R_{+}(\gamma)$ (resp. $R_{-}(\gamma)$) 
the union of those 
components of $R(\gamma)$ whose normal vectors point out of 
(resp. into) $M$. 
In this paper, 
we sometimes abbreviate $R_+(\gamma)$ (resp. $R_{-}(\gamma)$) 
to $R_{+}$ (resp. $R_{-}$). 
In the case that $(M,\gamma)$ is diffeomorphic 
to $(\Sigma \times [0,1], \partial \Sigma \times [0,1])$ 
where $\Sigma$ is a compact oriented surface, 
$(M, \gamma)$ is called a {\it product sutured manifold\/}.   
\end{definition}

Let $(M, i_{+}, i_{-}) \in \mathcal{C}_{g,n}$.
If we consider a small regular neighborhood of 
$i_{+} (\partial \Sigma_{g,n})=i_{-}(\partial \Sigma_{g,n})$ 
in $\partial M$ to be $\gamma$, 
we can regard $(M, i_{+}, i_{-})$ as a sutured manifold. 
However the converse is clearly not true in general. 
In the next section, we will determine which kinds of links give 
homology cylinders by considering their complementary 
sutured manifolds, which are defined as follows. 

\begin{definition}
Let $L$ be an oriented link in the 3-sphere $S^3$, 
and $\bar{R}$ a Seifert surface of $L$. 
Set $R:=\bar{R}\cap E(L)$, where $E(L)= {\rm cl }(S^3-N(L))$ is the complement 
of a regular neighborhood of $L$, and 
$(P,\delta):=(N(R,E(L)), N(\partial R,\partial E(L)))$. 
We call $(P, \delta)$ the {\it product sutured manifold for\/} $R$. 
Let $(M, \gamma)=({\rm cl }(E(L)-P), {\rm cl }(\partial E(L)-\delta))$ 
with $R_{\pm}(\gamma)=R_{\mp}(\delta)$. 
We call $(M, \gamma)$ the {\it complementary sutured manifold for\/} $R$. 
\end{definition}


\section{Homologically fibered links}\label{section:fibre}

Let $L$ be an oriented link in the 3-sphere $S^3$, and 
$\Delta_{L}(t)$ the normalized (one variable) 
Alexander polynomial of $L$, i.e., 
the lowest degree of $\Delta_{L}(t)$ is 0. 

\begin{definition}\label{def:HFKnot}
An $n$-component oriented link $L$ in $S^3$ is said to be 
{\it homologically fibered\/} if $L$ satisfies 
the following two conditions: 
\begin{enumerate}
\renewcommand{\labelenumi}{(\roman{enumi})}
\item The degree of $\Delta_{L}(t)$ is $2g+n-1$, 
where $g$ is the genus of a connected Seifert surface of $L$; and 
\item $\Delta_{L}(0)=\pm 1$. 
\end{enumerate}
An $n$-component oriented link $L$ satisfying (i) is 
said to be {\it rationally homologically fibered\/}. 
\end{definition}

Hereafter links are always assumed to be oriented. 
We also assume $2g+n-1 \ge 1$. Indeed 
the trivial knot is the only rationally homologically 
fibered link with $2g+n-1=0$. 

A link $L$ is said to be {\it fibered} if $E(L)$ is the total space of a 
fiber bundle over $S^1$ whose fiber is given by a Seifert surface. 
It is well known that fibered links satisfy the conditions 
in Definition \ref{def:HFKnot}. Hence they are homologically fibered. 

Let $L$ be an $n$-component link and $\Sigma_{g,n}$ 
the compact oriented surface 
that is diffeomorphic to a Seifert surface $R$ of $L$. 
We fix a diffeomorphism $\vartheta: 
\Sigma_{g,n} \stackrel{\cong}{\rightarrow} R$ 
and denote by $(M, \gamma)$ the complementary sutured manifold for $R$.
Then we may see that there are an orientation-preserving embedding 
$i_{+}: \Sigma_{g,n}\to \partial M$ 
and an orientation-reversing embedding 
$i_{-}: \Sigma_{g,n}\to \partial M$ with 
$i_{+}(\Sigma_{g,n})=R_{+}(\gamma)$ 
and $i_{-}(\Sigma_{g,n})=R_{-}(\gamma)$, where 
two embeddings $i_{\pm}$ are the composite maps of $\vartheta$ and 
the natural embeddings $\iota_{\pm}:R \hookrightarrow \partial M$: 
\[\SelectTips{cm}{}\xymatrix{
\Sigma_{g,n}  \ar[r]^{\vartheta}  \ar[dr]_{i_{\pm}} 
& R \ar[d]^{\iota_{\pm}} \\ & M }\]

If $i_{+},\,i_{-}: H_{1}(\Sigma_{g,n})\to H_{1}(M)$ are isomorphisms, 
we may regard $(M, \gamma)$ as a homology cylinder.  
The purpose of this section is to prove the next proposition.

\begin{proposition}\label{thm:homologicalfibre}
Let $R$ be a Seifert surface of a link $L$ with a diffeomorphism 
$\vartheta: \Sigma_{g,n} \stackrel{\cong}{\rightarrow} R$. 
If the complementary sutured manifold 
for $R$ is a $($rational$)$ homology cylinder,  
then $L$ is $($rationally$)$ homologically fibered. Conversely, 
if $L$ is $($rational$)$ homologically fibered, 
then the complementary sutured manifold 
for any minimal genus connected 
Seifert surface of $L$ gives a $($rational$)$ 
homology cylinder.
\end{proposition}

\begin{remark}
\begin{itemize}
\item[$(1)$] Aside from the name of homologically fibered links, 
the above fact was essentially mentioned in Crowell-Trotter \cite{ct}. 
\item[$(2)$] Suppose $L$ is a homologically fibered link and 
$M$ is the homology cylinder obtained from $L$ by the above procedure. 
If we change the diffeomorphism $\vartheta: \Sigma_{g,n} 
\xrightarrow{\cong} R$ into another 
one $\vartheta'$, then the resulting homology cylinder is 
$(\vartheta^{-1} \circ \vartheta')^{-1} \cdot M \cdot 
(\vartheta^{-1} \circ \vartheta') \in \mathcal{C}_{g,n}$, where 
$\vartheta^{-1} \circ \vartheta' \in \mathcal{M}_{g,n}$ is considered to 
be a homology cylinder as seen in Example \ref{ex:mgtocg}.
\end{itemize}
\end{remark}

For the proof of Proposition \ref{thm:homologicalfibre}, 
we first set up our notation, following \cite{bz} and \cite{lickorish}. 
Consider the basis 
$\{\alpha_{i}:=[\gamma_i]\}$ $(1\le i\le 2g+n-1)$ 
of $H_{1}(\Sigma_{g,n} ; \mathbb Z)\cong \mathbb Z^{2g+n-1}$ 
as shown in Figure \ref{fig:spine0}. 
We may see that $R$ consists of a disk $D^2$ and bands $B_{i}$
$(1\le i\le 2g+n-1)$, 
where the cores of $B_{i}$ correspond to $\vartheta(\alpha_{i})$. 
For simplicity, 
we use $\alpha_{i}$ again 
instead of $\vartheta(\alpha_{i})$. 
See Figure \ref{fig:trefoil} for the case of the trefoil.

\begin{figure}[h]
\centering
\includegraphics[width=.5\textwidth]{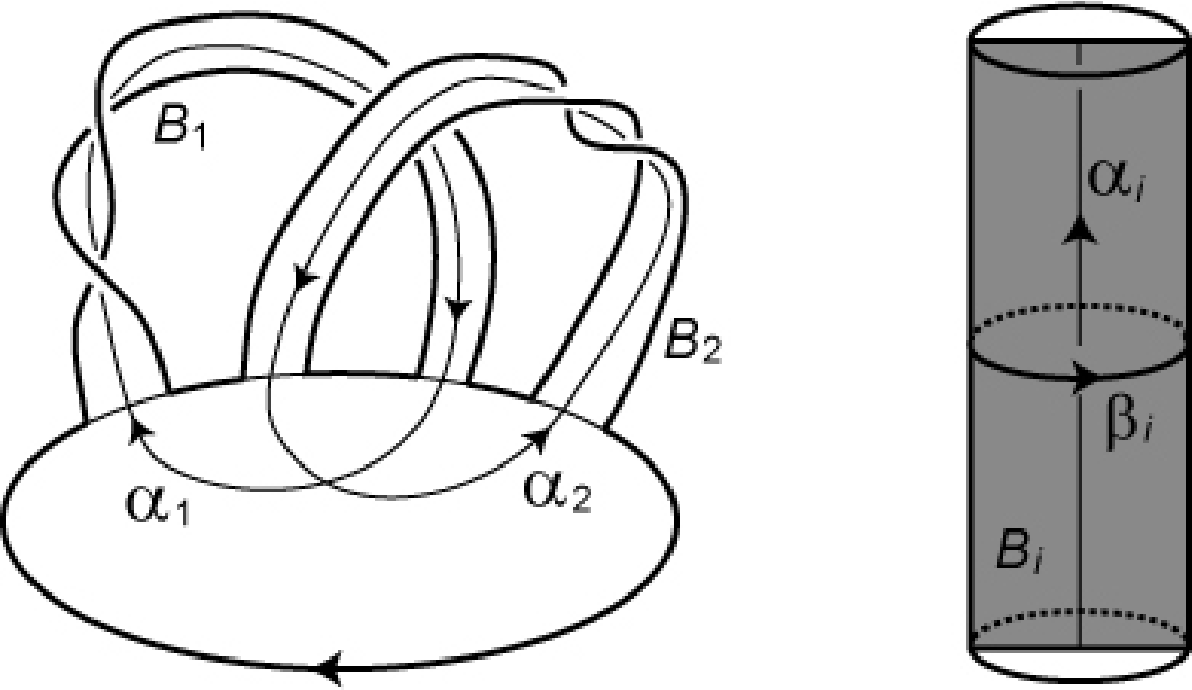}
\caption{Trefoil with the genus $1$ Seifert surface}\label{fig:trefoil}
\end{figure}

Let $(P,\delta)$ be the product sutured manifold for $R$.
The curves $\alpha_{1}, \ldots , \alpha_{2g+n-1}$ of $R$ are 
projected onto curves $\alpha_{1}^{+}, \ldots , \alpha_{2g+n-1}^{+}$ 
on $R_{+}(\delta)$ by $\iota_{+}$, 
and  $\alpha_{1}^{-}, \ldots , \alpha_{2g+n-1}^{-}$ on $R_{-}(\delta)$ 
by $\iota_{-}$. 
Choose a curve $\beta_{i}$ on the boundary of the regular neighborhood 
of the band $B_{i}$
so that each $\beta_{i}$ bounds a disk in $P$ 
that meets $\alpha_{i}$ at one point.
The orientation of the disk and of $\beta_{i}$ 
are chosen such that the intersection number is +1. 
(See Figure \ref{fig:trefoil}, or \cite[Figure 8.3]{bz}.) 

\begin{lemma}\label{lem:basis}
\begin{itemize}
\item[$(1)$] 
The set
$\{\alpha_{1}^{\varepsilon}, \ldots , \alpha_{2g+n-1}^{\varepsilon}, 
\beta_1,\ldots , \beta_{2g+n-1} \}$ 
with $\varepsilon = +1  \text{ or } -$ is a basis of 
$H_1 (\partial M)=H_{1}(\partial P)\cong \mathbb Z^{4g+2n-2}$. 
\item[$(2)$] 
$\{\alpha_{1}^{\varepsilon}, \ldots , \alpha_{2g+n-1}^{\varepsilon}\}$ 
with $\varepsilon = +1  \text{ or } -$ 
is a basis of $H_{1}(P)$ and 
$\{\beta_{1}, \ldots , \beta_{2g+n-1}\}$ is a basis of 
$H_{1}(M)\cong \mathbb Z^{2g+n-1}$. 
\item[$(3)$] 
$H_\ast (M)=0$ for $\ast \ge 2$.
\end{itemize}
\end{lemma}

\begin{proof}
It is not difficult to show (1) and the first statement in (2). 
For the second one in (2), one may apply the Mayer-Vietoris 
sequence: 
\[0=H_{2}(S^3)\rightarrow H_{1}(\partial M)\xrightarrow{\phi}
H_{1}(P)\oplus H_{1}(M)\rightarrow H_{1}(S^3)=0.\]
Note that $\partial M = \partial P$ and 
$\phi(\beta_{i})=(0,\beta_{i})$. 
Then, the conclusion follows from (1).

In the exact sequence 
$H_{1}(\partial M) \to H_{1}(M) \to H_{1}(M,\partial M) 
\to 0$, 
the first map is surjective from (1) and (2). 
Thus $H_{1}(M,\partial M)=0$. 
By the Poincar\'e duality, 
we have $H_{2}(M) \cong H^1 (M,\partial M)=0$. 
Clearly $H_\ast (M)=0$ for $\ast \ge 3$, and (3) holds. 
\end{proof}

Let $\mathcal{S}$ be the Seifert matrix 
corresponding to the above basis of $H_1 (R)$, 
namely 
$\mathcal{S}=(a_{jk})=(\text{lk}(\alpha^{-}_{j}, \alpha_{k}))\,
(1\le j,\,k\le 2g+n-1)$.

\begin{lemma}\label{lem:inclusion}
Let $\iota_{\pm} : R_{\pm}(\delta)\to M$ 
denote the inclusions. Then, 
\[\iota_{+}(\alpha^{+}_{j})=\sum_{k=1}^{2g+n-1}a_{kj}\beta_{k}
\hspace{0.5cm}\text{ and }\hspace{0.5cm} 
\iota_{-}(\alpha^{-}_{j})=\sum_{k=1}^{2g+n-1}a_{jk}\beta_{k}.\]
\end{lemma}

\begin{proof}
See the proof of \cite[Lemma 8.6]{bz} or \cite[Page 53]{lickorish}. 
\end{proof}

By Lemma \ref{lem:inclusion}, we have: 

\begin{lemma}\label{lem:seifertmatrix1}
The maps $i_{\pm} : H_{1}(\Sigma_{g,n})\to H_{1}(M)$ 
$($resp. $i_{\pm} : H_{1}(\Sigma_{g,n};\mathbb Q)\to 
H_{1}(M;\mathbb Q))$ are isomorphisms 
if and only if $\mathcal{S}$ is invertible over $\mathbb Z$ 
$($resp. over $\mathbb Q)$. 
\end{lemma}

\medskip
\begin{proof}[Proof of Proposition $\ref{thm:homologicalfibre}$]
Suppose that the complementary sutured manifold $M$ for 
$R$ is a rational homology cylinder. 
Then $\mathcal{S}$ is invertible over $\mathbb{Q}$ 
by Lemma \ref{lem:seifertmatrix1} 
and $(\mathcal{S}^{T})^{-1} \mathcal{S}$ represents 
$\sigma (M)$, where 
$\mathcal{S}^{T}$ denotes the transpose of $\mathcal{S}$. 
By definition, we have 
$\Delta_{L}(t)=\det(t \mathcal{S}-\mathcal{S}^{T})$, 
and now 
\begin{equation}\label{eq:factor}
\Delta_{L}(t)=\det(t \mathcal{S}-\mathcal{S}^{T})
=\det(\mathcal{S}^{T}) 
\det(t (\mathcal{S}^{T})^{-1} \mathcal{S}-I_{2g+n-1})
\end{equation}
\noindent
holds. Since $\det((\mathcal{S}^{T})^{-1} \mathcal{S})=1$, 
the polynomial $\det((t (\mathcal{S}^{T})^{-1} \mathcal{S}-I_{2g+n-1})$ 
is of degree $2g+n-1$ and so is $\Delta_{L}(t)$. 
Therefore $L$ is rationally homologically fibered. 
If moreover $M$ is a homology cylinder, 
then we have $\det(\mathcal{S})=\pm 1$ and 
$\Delta_{L}(0)=\det(-\mathcal{S}^T)=\pm 1$. Hence 
$L$ is homologically fibered.

Conversely, let $L$ be a rationally homologically fibered link and 
$R$ be a minimal genus, say $g$, connected Seifert surface. 
Then, the degree of 
$\Delta_{L}(t)$ is $2g+n-1$. 
Since $\Delta_{L}(t)=\det(t\mathcal{S}-\mathcal{S}^{T})$ 
and $0\neq \Delta_{L}(0)=\det(-\mathcal{S}^{T})$, 
the complementary 
sutured manifold for $R$ is a rational homology cylinder 
by Lemma \ref{lem:seifertmatrix1}. 
Further, if $L$ is homologically fibered, we have 
$\pm 1=\Delta_{L}(0)=\det(-\mathcal{S}^{T})=\det(-\mathcal{S})$. 
This completes the proof. 
\end{proof}

\begin{example}\label{rmk:alternatinglink} 
It is known (\cite{crowell}, \cite{murasugi58}) that 
alternating links satisfy the condition (i) in Definition \ref{def:HFKnot}. 
Moreover it was shown by Murasugi 
\cite{murasugi} (see also 13.26 (c) in \cite{bz}) that 
an alternating link is fibered if and only if 
$\Delta_{L}(0)=\pm 1$. 
Therefore, if a homologically fibered link $L$ is not 
fibered, then it is non-alternating. 
\end{example}

\begin{example}\label{ex:pretzel1}
Let $p,\,q$ and $r$ be odd integers and let $P(p,q,r)$ 
be the pretzel knot  of type $\{p,q,r\}$. See Figure \ref{pretzel1}. 
We assume that one of $p,\,q,\,r$, say $p$, is negative and 
the others are positive since 
our main objects are non-alternating knots 
(Example \ref{rmk:alternatinglink}). 
It is well-known that 
the Alexander polynomial of $P(p,q,r)$ is given by 
\[\frac{1}{4}\Big((pq+qr+rp)(t^2-2t+1)+t^2+2t+1\Big).\]
In the range of values: $-100<p\le -3,\, 3\le q\le r<100$, 
the pretzel knots of the following 22 types are homologically fibered knots. 
\begin{align*}
&\{-3,5,9\},\{-5,7,19\},\{-7,9,33\},\{-9,11,51\},\{-9,15,23\},\{-11,13,73\},\\
&\{-13,15,99\},\{-15,21,53\},\{-19,33,45\},\{-21,27,95\},\{-23,37,61\},\\
&\{-33,59,75\},\{-3,5,5\},\{-5,7,15\},\{-7,9,29\},\{-9,11,47\},\{-11,13,69\},\\
&\{-13,15,95\},\{-15,25,37\},\{-25,35,87\},\{-29,51,67\},\{-37,59,99\}.
\end{align*}
\noindent
The minimal genus (genus 1) Seifert surface for the pretzel knot of this type 
is unique up to isotopy \cite{gi}. 
\end{example}

\begin{figure}[h]
\centering
\includegraphics[width=.7\textwidth]{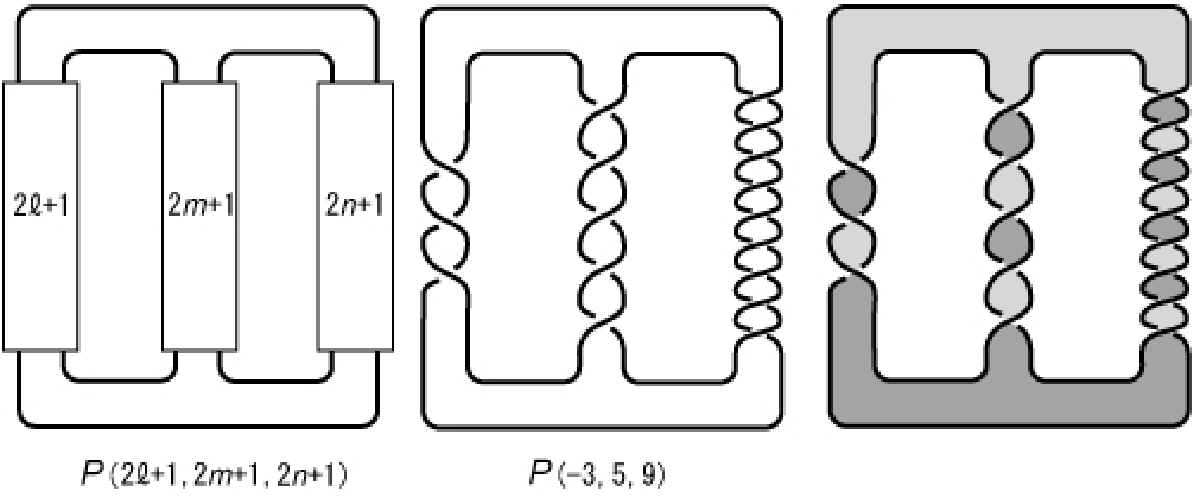}
\caption{Standard diagram of Pretzel knots}\label{pretzel1}
\end{figure}

\begin{example} 
Consider the pretzel knot of type $\{p,q,r,s,u\}$, where 
$p,\,q,\,r,\,s,\,u$ are odd integers. 
The leading coefficient of the Alexander polynomial is
\[\frac{1}{16}(pq+pr+ps+pu+qr+qs+qu+rs+ru+su+pqrs+pqru+pqsu+prsu+qrsu).\]
In the range of values: $-500<p\le -3,\,3\le q\le r\le s\le u<500$, 
the following 8 types give the homologically fibered pretzel knots. 
\begin{align*}
&\{-3,9,9,9,85\},\{-5,15,15,15,411\},\{-7,17,17,45,261\},\\
&\{-9,15,35,71,467\},\{-33,75,127,151,403\},\{-39,113,161,165,221\},\\
&\{-9,23,27,35,411\},\{-37,107,107,179,363\}.
\intertext{In the range of values: 
$-300<p\le q\le -3,\,3\le r\le s\le u<300$, 
the following 15 types give the homologically fibered pretzel knots.}
&\{-15,-3,5,5,125\},\{-5,-5,3,19,159\},\{-69,-5,7,15,151\},\\
&\{-31,-7,9,17,177\},\{-27,-11,9,85,205\},\{-15,-3,5,5,129\},\\
&\{-5,-5,3,19,163\},\{-53,-5,7,15,91\},\{-177,-5,7,31,31\},\\
&\{-257,-5,7,19,99\},\{-235,-7,17,17,33\},\{-15,-11,13,13,265\},\\
&\{-275,-11,13,109,117\}, \{-37,-33,23,111,207\},\{-121,-33,39,107,279\}.
\end{align*}
\end{example}

\begin{example}\label{ex:gkm}
Let $K$ be the trefoil knot, which is fibered. 
We take the basis $\{\alpha_{1}, \alpha_{2}\}$ of $H_1 (R)$ 
for the minimal genus Seifert surface $R$ as in Figure \ref{fig:3-1}. 
We cut the band corresponding to $\alpha_{2}$, make it knotted, 
and paste to the original part again, then we have a new knot 
with a Seifert surface of the same genus. 
Just before pasting, we twist the band so that 
the Seifert matrix (linking number) does not change, then 
we can obtain a knot whose Alexander polynomial is the same as $K$. 
By this method, we can obtain many homologically fibered knots.
\end{example}

\begin{figure}[h]
\centering
\includegraphics[width=.85\textwidth]{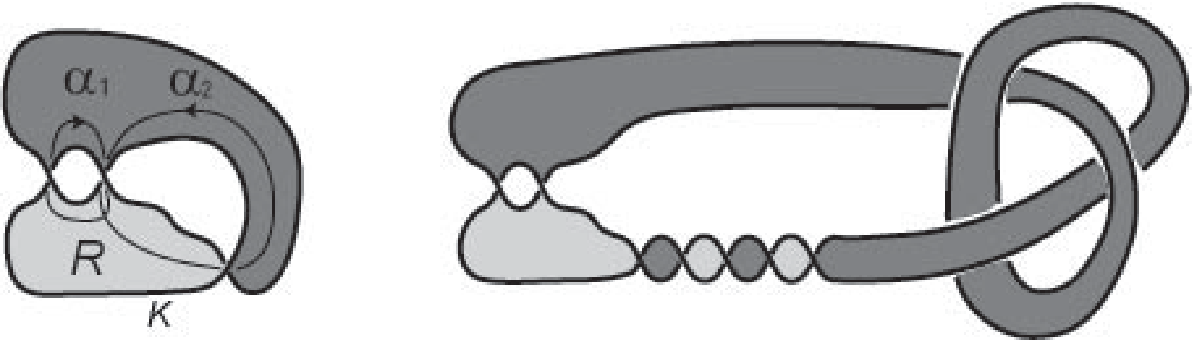}
\caption{Making a new homologically fibered knot}\label{fig:3-1}
\end{figure}

\begin{example}\label{ex:FK}
It is known that a knot $K$ with 11 or fewer crossings is fibered 
if and only if 
$K$ is homologically fibered. 
Among 12 crossing knots there are thirteen knots  
which are not fibered but homologically fibered.  
See Friedl-Kim\cite{fk} for the detail.
\end{example}


\section{Invariants of homology cylinders and 
fibering obstructions of links}\label{sec:magnus}
In this section, 
we review some invariants of homology cylinders from \cite{sakasai}. 
We begin by summarizing our notation. 
For a matrix $A$ with entries in a ring $\mathcal{R}$, 
and a ring homomorphism $\rho:\mathcal{R} \to \mathcal{R}'$, 
we denote by ${}^{\rho} A$ the 
matrix obtained from $A$ by applying $\rho$ to each entry. 
When $\mathcal{R}=\mathbb{Z} G$ (or its fractional field if it exists) 
for a group $G$, 
we denote by $\overline{A}$ 
the matrix obtained from $A$ by applying the involution induced 
from $(x \mapsto x^{-1},\ x \in G)$ to each entry. 
For a module $\mathcal{M}$, we write $\mathcal{M}^n$ 
for the module of column vectors with $n$ entries. 
For a finite cell complex $X$, we denote by $\widetilde{X}$ 
its universal covering. We take a base point $p$ of $X$. 
The group $\pi:=\pi_1 (X,p)$ acts on $\widetilde{X}$ from 
the {\it right} as its deck transformations. 
Then the cellular chain complex $C_{\ast} (\widetilde{X})$ of 
$\widetilde{X}$ 
becomes a right $\mathbb{Z} \pi$-module. 
For each left $\mathbb{Z}\pi$-algebra $\mathcal{R}$, 
the twisted chain complex $C_{\ast} (X;\mathcal{R})$ is given 
by the tensor product of 
the right $\mathbb{Z}\pi$-module $C_{\ast} (\widetilde{X})$ and 
the left $\mathbb{Z}\pi$-module $\mathcal{R}$, 
so that $C_{\ast} (X;\mathcal{R})$ and 
$H_{\ast} (X;\mathcal{R})$ are right $\mathcal{R}$-modules.

Let $M=(M,i_+,i_-) \in \mathcal{C}_{g,n}^\mathbb{Q}$ 
and let $\rho_\Gamma:\pi_1 (M) \to \Gamma$ be a homomorphism 
whose target $\Gamma$ 
is a {\it poly-torsion-free abelian $($PTFA$)$ group}, where 
a group $\Gamma$ is said to be PTFA if 
it has a sequence 
\[\Gamma=\Gamma_0 \triangleright \Gamma_1 \triangleright 
\cdots \triangleright \Gamma_n = \{1\}\]
whose successive quotients $\Gamma_i/\Gamma_{i+1}$ $(i \ge 0)$ 
are all torsion-free abelian. 
Using a PTFA group $\Gamma$ has an advantage that 
its group ring $\mathbb{Z} \Gamma$ (or $\mathbb{Q} \Gamma$) 
is an {\it Ore domain} so that it is embedded into 
the {\it right field} 
\[\mathcal{K}_\Gamma:= \mathbb{Z}\Gamma (\mathbb{Z}\Gamma - \{0\})^{-1}
=\mathbb{Q}\Gamma (\mathbb{Q}\Gamma - \{0\})^{-1}\] 
{\it of fractions}. We refer to Cochran-Orr-Teichner 
\cite[Section 2]{cot} and Passman \cite{pa} 
for generalities of PTFA groups and localizations of their 
group rings. 
A typical example of PTFA groups associated with $M$ is 
the free part $\Gamma = H_1 (M)/\mbox{(torsion)}\cong \mathbb{Z}^{2g+n-1}$ 
of $H_1 (M)$, 
where $\mathcal{K}_\Gamma$ 
is isomorphic to the field of rational functions with $2g+n-1$ variables. 
The following lemma can be verified by applying Cochran-Orr-Teichner 
\cite[Proposition 2.10]{cot}. However we here give a proof for later use. 
\begin{lemma}\label{lem:relative}
The maps $i_\pm: H_\ast (\Sigma_{g,n},p ;i_{\pm}^\ast \mathcal{K}_\Gamma) \to 
H_\ast (M,p ;\mathcal{K}_\Gamma)$ are isomorphisms as 
right $\mathcal{K}_\Gamma$-vector spaces. 
\end{lemma}
\begin{proof}
For the proof, it suffices to show that 
$H_\ast (M, i_+(\Sigma_{g,n}) ;\mathcal{K}_\Gamma)=0$. 
Since the spine $S$ fixed in Section \ref{section:cylinder} 
is a deformation retract of $\Sigma_{g,n}$ relative to $p$, 
we have $H_\ast (M, i_+(\Sigma_{g,n}) ;\mathcal{K}_\Gamma) = 
H_\ast (M, i_+(S) ;\mathcal{K}_\Gamma)$. Now we compute the latter. 

Triangulate $\Sigma_{g,n}$ smoothly, so that the spine $S$ 
is the union of its edges. 
By gluing two copies of this triangulated surface, 
we obtain a triangulation $\mathfrak{t}$ of $\partial M$. 
A theorem of Cairns and Whitehead shows that 
there exists a triangulation $\widehat{\mathfrak{t}}$ of the entire $M$ 
which extends $\mathfrak{t}$. 
Starting from a 2-simplex in $\partial M$, we can deform $M$ 
onto a subcomplex $\widehat{\mathfrak{t}}$ of its 2-skeleton. 
In this deformation, the 1-skeleton is fixed pointwise. 
Take a maximal tree $T$ 
of $\mathfrak{t}$ such that 
$T$ includes all but one sub-edges of each loop of $S$. 
We extend $T$ to a maximal tree $\widetilde{T}$ of 
$\widetilde{\mathfrak{t}}$ and collapse $\widetilde{T}$ to a point. 
Then we obtain a 2-dimensional CW-complex $M'$ having 
only one vertex. By construction, the bouquet $i_+(S)$ is 
mapped onto a bouquet $S'$ in $M'$ with a natural 
one-to-one correspondence between their loops, 
and $(M',S')$ is simple 
homotopy equivalent to $(M,i_+(S))$. 
>From this cell structure, we can read a presentation 
of $\pi_1 (M) = \pi_1 (M')$ as 
\begin{equation}\label{eq:pre_adm}\langle y_1 ,\ldots, y_k, 
i_+ (\gamma_1),\ldots,i_+ (\gamma_{2g+n-1}) \mid 
s_1, \ldots, s_k
\rangle
\end{equation}
for some $k$, where we identify $i_+(\gamma_j)$ $(1 \le j \le 2g+n-1)$ 
with its image in $M'$. 

We have 
$H_\ast (M, i_+(S) ;\mathcal{K}_\Gamma)=
H_\ast (M', S';\mathcal{K}_\Gamma)$. 
The relative complex $(M',S')$ 
consists of only the same number of 1-cells and 2-cells, so that 
the relative chain complex 
$C_\ast (M',S';\mathcal{K}_\Gamma)$ is 
of the form 
\[0 \longrightarrow (\mathcal{K}_\Gamma)^k 
\stackrel{{}^{\rho_\Gamma}J \cdot}{\longrightarrow} 
(\mathcal{K}_\Gamma)^k \longrightarrow 0\]
with $J:= \overline{\left({\displaystyle\frac{\partial 
s_j}{\partial y_i}} \right)}_{1 \le i,j \le k}$. 
The matrix ${}^{\rho_\Gamma} J$ has its entries in $\mathbb{Z} \Gamma$. 
To check the invertibility over $\mathcal{K}_\Gamma$ of this matrix, 
we apply the augmentation map $\mathfrak{a}:\mathbb{Z} \Gamma 
\to \mathbb{Z}$ to each entry. 
Then we obtain a presentation matrix of $H_1 (M,i_+(\Sigma_{g,n}))$. 
Since $H_1 (M,i_+(\Sigma_{g,n});\mathbb{Q})=0$, the matrix 
${}^{\mathfrak{a} \circ \rho_\Gamma} J$ is invertible over $\mathbb{Q}$. 
Then it follows from Strebel \cite[Section 1]{strebel} that 
${}^{\rho_\Gamma} J$ is invertible over $\mathcal{K}_\Gamma$. 
($\Gamma$ belongs to the class $D(\mathbb{Z})$ 
in the notation of \cite{strebel}.) 
This completes the proof 
\end{proof}

We use Lemma \ref{lem:relative} to construct the following two invariants 
of rational homology cylinders. 
The first one is the Magnus matrix, which was defined in
\cite{sakasai}. We have
\[H_1 (\Sigma_{g,n},p;i_{\pm}^\ast \mathcal{K}_\Gamma) \cong 
H_1 (S,p;i_{\pm}^\ast \mathcal{K}_\Gamma) = 
C_1 (\widetilde{S}) \otimes_{\pi_1 (\Sigma_{g,n})} 
i_{\pm}^\ast \mathcal{K}_\Gamma \cong \mathcal{K}_\Gamma^{2g+n-1}\]
with a basis
\[\{ \widetilde{\gamma}_1 \otimes 1, \ldots , 
\widetilde{\gamma}_{2g+n-1} \otimes 1\} 
\subset C_1 (\widetilde{S}) \otimes_{\pi_1 (\Sigma_{g,n})} 
i_{\pm}^\ast \mathcal{K}_\Gamma\]
\noindent 
as a right $\mathcal{K}_\Gamma$-module. 
Here we fix a lift $\widetilde{p}$ of $p$ as a 
base point of $\widetilde{S}$, and 
denote by $\widetilde{\gamma}_i$ the lift of 
the oriented loop $\gamma_i$. 

\begin{definition}\label{def:Mag2}
For $M=(M,i_+,i_-) \in \mathcal{C}_{g,n}^\mathbb{Q}$, 
the {\it Magnus matrix} 
\[r_\Gamma (M) \in GL(2g+n-1,\mathcal{K}_\Gamma)\]
of $M$ is defined as the representation matrix of 
the right $\mathcal{K}_\Gamma$-isomorphism
\[\mathcal{K}_\Gamma^{2g+n-1} \cong 
H_1 (\Sigma_{g,n},p;i_-^\ast \mathcal{K}_\Gamma) \xrightarrow[i_-]{\cong} 
H_1 (M,p;\mathcal{K}_\Gamma) \xrightarrow[i_+^{-1}]{\cong} 
H_1 (\Sigma_{g,n},p;i_+^\ast\mathcal{K}_\Gamma) 
\cong \mathcal{K}_\Gamma^{2g+n-1},\]
where the first and the last isomorphisms use 
the bases mentioned above.
\end{definition}

\begin{example}\label{fox}
For $(\Sigma_{g,n} \times [0,1], 
\mathrm{id} \times 1, \varphi \times 0) \in \mathcal{M}_{g,n} \subset 
\mathcal{C}_{g,n}$, 
we can check that 
\[r_\Gamma((\Sigma_{g,n} \times [0,1], 
\mathrm{id} \times 1, \varphi \times 0)) = 
\overline{
\sideset{^{\rho_\Gamma}\!}{}
{\mathop{\left({\displaystyle\frac{\partial 
\varphi(\gamma_j)}{\partial \gamma_i}} \right)}\nolimits}
}_{1 \le i,j \le 2g+n-1}\]
from the definition or by using Proposition \ref{prop:Formula} below. 
>From this, we see that 
$r_\Gamma$ extends the Magnus representation of $\mathcal{M}_{g,1}$ 
in Morita \cite{morita}. 
\end{example}

Next we introduce a torsion invariant. Since the relative complex 
$C_\ast (M,i_+(\Sigma_{g,n});\mathcal{K}_\Gamma)$ 
obtained from any cell decomposition of $(M,i_+(\Sigma_{g,n}))$ 
is acyclic by Lemma \ref{lem:relative}, 
we can consider its torsion 
$\tau(C_\ast (M,i_+(\Sigma_{g,n});$ $\mathcal{K}_\Gamma))$. 
We refer to Milnor \cite{milnor} and 
Turaev \cite{turaev} for generalities of torsions and 
related groups from algebraic K-theory. 
Recall that torsions are invariant 
under simple homotopy equivalences. In particular, they are 
topological invariants. 
\begin{definition}
The $\Gamma$-{\it torsion} of $M=(M,i_+,i_-) 
\in \mathcal{C}_{g,n}^\mathbb{Q}$ is given by 
\[\tau_{\Gamma}^+ (M):=
\tau(C_\ast (M,i_+(\Sigma_{g,n});\mathcal{K}_\Gamma))
\in K_1 (\mathcal{K}_\Gamma)/\pm \rho_\Gamma(\pi_1 (M)).\]
\end{definition}

Now we recall a method for computing 
$r_\Gamma(M)$ and $\tau_{\Gamma}^+ (M)$ by 
following \cite[Section 3.2]{sakasai}, 
which is based on 
the one for the Gassner matrix (using commutative rings) 
of a string link by Kirk-Livingston-Wang \cite{klw} and 
Le Dimet \cite[Section 1.1]{ld}. 

Let $(M,i_+,i_-) \in \mathcal{C}_{g,n}^\mathbb{Q}$. 
An {\it admissible presentation} of $\pi_1 (M)$ is defined to be 
the one of the form 
\begin{equation}\label{admissible}
\langle i_- (\gamma_1),\ldots,i_- (\gamma_{2g+n-1}), 
z_1 ,\ldots, z_l, 
i_+ (\gamma_1),\ldots,i_+ (\gamma_{2g+n-1}) \mid 
r_1, \ldots, r_{2g+n-1+l}
\rangle
\end{equation}
for some integer $l$. 
That is, it is a finite presentation with deficiency $2g+n-1$ 
whose generating set 
contains $i_- (\gamma_1),\ldots,i_- (\gamma_{2g+n-1}), 
i_+ (\gamma_1),\ldots,i_+ (\gamma_{2g+n-1})$ and is ordered as above. 
One of the possible constructions 
of admissible presentations is obtained from 
the presentation (\ref{eq:pre_adm}) by 
adding generators $i_- (\gamma_1), \ldots, 
i_- (\gamma_{2g+n-1})$ 
together with relations. 
(There also exists a construction using Morse theory.) 

Given an admissible presentation of $\pi_1 (M)$ 
as in (\ref{admissible}), 
we define $(2g+n-1) \times (2g+n-1+l)$, $l \times (2g+n-1+l)$ and 
$(2g+n-1) \times (2g+n-1+l)$ matrices $A,B,C$ 
over $\mathbb{Z} \pi_1 (M)$ by 
\[A=\overline{
\left(\frac{\partial r_j}{\partial i_-(\gamma_i)}
\right)}_{\begin{subarray}{c}
{}1 \le i \le 2g+n-1\\
1 \le j \le 2g+n-1+l
\end{subarray}}, \ 
B=\overline{
\left(\frac{\partial r_j}{\partial z_i}
\right)}_{\begin{subarray}{c}
{}1 \le i \le l\\
1 \le j \le 2g+n-1+l
\end{subarray}}, \ 
C=\overline{
\left(\frac{\partial r_j}{\partial i_+(\gamma_i)}
\right)}_{\begin{subarray}{c}
{}1 \le i \le 2g+n-1\\
1 \le j \le 2g+n-1+l
\end{subarray}}.\]

\begin{proposition}\label{prop:Formula}
As matrices with entries in $\mathcal{K}_\Gamma$, we have the following. 

\begin{itemize}
\item[$(1)$] The square matrix 
$\sideset{^{\rho_\Gamma}\!\!}{}
{\mathop{\begin{pmatrix} A \\ B \end{pmatrix}}\nolimits}$ 
is invertible and 
$\tau_{\Gamma}^+ (M)=\sideset{^{\rho_\Gamma}\!\!}{}
{\mathop{\begin{pmatrix} A \\ B \end{pmatrix}}\nolimits}$. 

\item[$(2)$] $r_\Gamma(M) = 
-{}^{\rho_\Gamma}C 
\sideset{^{\rho_\Gamma}\!\!}{^{-1}}
{\mathop{\begin{pmatrix} A \\ B \end{pmatrix}}\nolimits}
\begin{pmatrix} I_{2g+n-1} \\ 0_{(l,2g+n-1)}\end{pmatrix}$. 
\end{itemize}
In particular, the invariants $\tau_{\Gamma}^+ (M)$ and 
$r_\Gamma(M)$ are computable 
from any admissible presentation of $\pi_1 (M)$. 
\end{proposition}
\begin{proof}
(1) For an admissible presentation of 
$\pi_1 (M)=\pi_1 (M')$ obtained 
from (\ref{eq:pre_adm}), 
the torsion $\tau_{\Gamma}^+ (M)$ is given 
by the matrix ${}^{\rho_\Gamma} J$. 
Hence our claim holds in this case. 

Given any admissible presentation $P$ of $\pi_1 (M)$ 
as in (\ref{admissible}), we construct a 
2-complex $X(P)$ having one 0-cell as a basepoint, 
$(4g+2n-2+l)$ 1-cells 
indexed by the generators and 
$(2g+n-1+l)$ 2-cells indexed by the relations 
and attached according to the words. 
Then we can use a theorem of Harlander-Jensen \cite[Theorem 3]{hj} 
with the fact that the deficiency of $\pi_1 (M)$ is $2g+n-1$ 
(see Epstein \cite[Lemmas 1.2, 2,2]{ep}) to show that 
$X(P)$ and $M'$ are homotopy equivalent. In fact, 
there exists a basepoint preserving 
cellular map $f:X(P) \to M'$ 
which is a homotopy equivalence and 
maps the union $S_0$ of the 1-cells of $P_0$ 
corresponding to $i_+(\gamma_1),\ldots,
i_+(\gamma_{2g+n-1})$ homeomorphically onto $S'$. 
Let $M_f$ be the mapping cylinder of $f$. 
We have 
\begin{align*}
\tau_{\Gamma}^+ (M) &= \tau(C_\ast (M,i_+(\Sigma_{g,n});\mathcal{K}_\Gamma))
=\tau(C_\ast (M,i_+(S);\mathcal{K}_\Gamma))\\
&=\tau(C_\ast (M',S';\mathcal{K}_\Gamma))
=\tau(C_\ast (M_f,S';\mathcal{K}_\Gamma))
=\tau(C_\ast (M_f,S_0 \times [0,1];\mathcal{K}_\Gamma))\\
&=\tau(C_\ast (M_f,S_0;\mathcal{K}_\Gamma))
=\tau(C_\ast (M_f,X(P);\mathcal{K}_\Gamma)) 
\tau(C_\ast (X(P),S_0;\mathcal{K}_\Gamma))
\end{align*}
\noindent
where we repeatedly used the multiplicativity of torsions. 
(For example, we have
\[\tau(C_\ast (M,i_+(S);\mathcal{K}_\Gamma))
=\tau(C_\ast (M,i_+(\Sigma_{g,n});\mathcal{K}_\Gamma))
\tau(C_\ast (i_+(\Sigma_{g,n}),i_+(S);\mathcal{K}_\Gamma))\]
with $\tau(C_\ast (i_+(\Sigma_{g,n}),i_+(S);\mathcal{K}_\Gamma))=1$ 
since $i_+(\Sigma_{g,n})$ is simple homotopy 
equivalent to $i_+(S)$.) 

We now compute $\tau(C_\ast(X(P),S_0;\mathcal{K}_\Gamma))$. 
As in the case of the complex $(M',S')$, 
the relative complex $(X(P), S_0)$ 
consists of only the same number of 1-cells and 2-cells. 
Thus $\tau(C_\ast (X(P),S_0;\mathcal{K}_\Gamma))$ is given by 
$\sideset{^{\rho_\Gamma}\!\!}{}
{\mathop{\begin{pmatrix} A \\ B \end{pmatrix}}\nolimits}$, 
which is a square matrix over $\mathbb{Z} \Gamma$. 
By an argument similar to the matrix $J$ in the proof of 
Lemma \ref{lem:relative}, we can check 
that this matrix is invertible over $\mathcal{K}_\Gamma$. 

If $M$ is an irreducible 3-manifold, 
it is a Haken manifold since $|H_1 (M)|=\infty$. 
Waldhausen's theorem 
\cite[Theorems 19.4, 19.5]{waldhausen1} shows that 
the Whitehead group $Wh(\pi)=K_1 (\mathbb{Z} \pi_1 (M))/\pm \pi_1 (M)$ 
of $\pi_1 (M)$ vanishes. Hence $X(P)$, $M'$ and $M_f$ are 
simple homotopy equivalent and 
we have $\tau(C_\ast (M_f,X(P);\mathcal{K}_\Gamma)) =1$. 
The second claim of (1) follows in this case. 

If $M$ is not irreducible, we can check that 
$M$ is a connected sum of 
a Haken manifold $M_0$ containing $\partial M$ 
and a (possibly reducible) rational homology 3-sphere $M_2$. 
Since any homomorphism from 
$\pi_1 (M_2)$ to a PTFA group $\Gamma$ is trivial, the homomorphism 
$\rho_\Gamma$ factors through $\pi_1 (M_1)$, whose 
Whitehead group vanishes as mentioned above. 
Now $\tau(C_\ast (M_f,X(P);\mathcal{K}_\Gamma))$ is the image of 
the Whitehead torsion 
$\tau(C_\ast (M_f,X(P);\mathbb{Z} \pi_1 (M))) \in Wh(\pi_1 (M))$ 
by $\rho_\Gamma$. It must be trivial 
since it passes through $Wh (\pi_1 (M_1))=0$. This completes the proof. 

(2) The proof is almost identical to 
that in \cite[Proposition 3.9]{sakasai}, 
and here we omit it. 
\end{proof}

The $\Gamma$-torsion and the Magnus matrix 
can be used as fibering obstructions of 
a homologically fibered link as follows. 
If a link is fibered, 
the complementary sutured manifold for each minimal genus Seifert 
surface is a product sutured manifold, 
whose $\Gamma$-torsion is trivial for any $\mathcal{K}_\Gamma$. 
Together with Example \ref{fox}, we have: 

\begin{theorem}\label{thm:fiberness}
\begin{itemize}
\item[$(1)$] 
Suppose a homologically fibered link has 
a minimal genus Seifert surface which 
gives a homology cylinder having non-trivial $\Gamma$-torsion for 
some PTFA group $\Gamma$, then it is not fibered. 
\item[$(2)$] 
Let $M$ be a homology cylinder obtained from a minimal 
genus Seifert surface 
of a fibered link. 
Then all the entries of the Magnus matrix $r_{\Gamma}(M)$ are 
in $\mathbb{Z} \Gamma$. 
\end{itemize}
\end{theorem}

\begin{example}
Let $K=P(-3,5,9)$, which is a homologically fibered knot as seen in 
Example \ref{ex:pretzel1}. We take a Seifert surface of $K$ and its spine 
as in Figure \ref{fig:pretzel2}, where the darker color means the $+$-side.

\begin{figure}[htbp]
\begin{minipage}{.45\textwidth}
\begin{center}
\includegraphics[width=.7\textwidth]{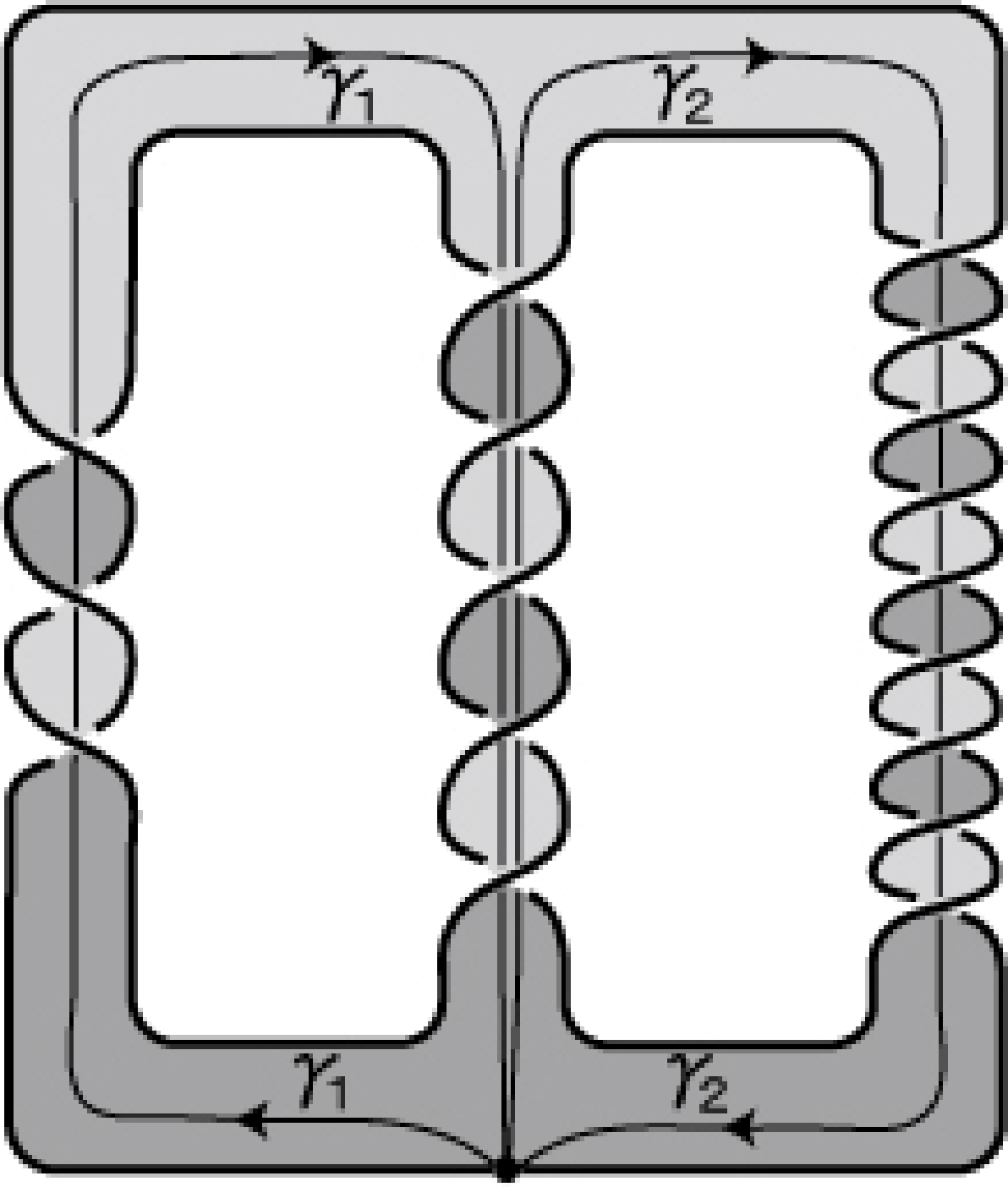}
\end{center}
\caption{A Seifert surface of $P(-3,5,9)$ and its spine}
\label{fig:pretzel2}
\end{minipage}
\begin{minipage}{.45\textwidth}
\begin{center}
\includegraphics[width=.7\textwidth]{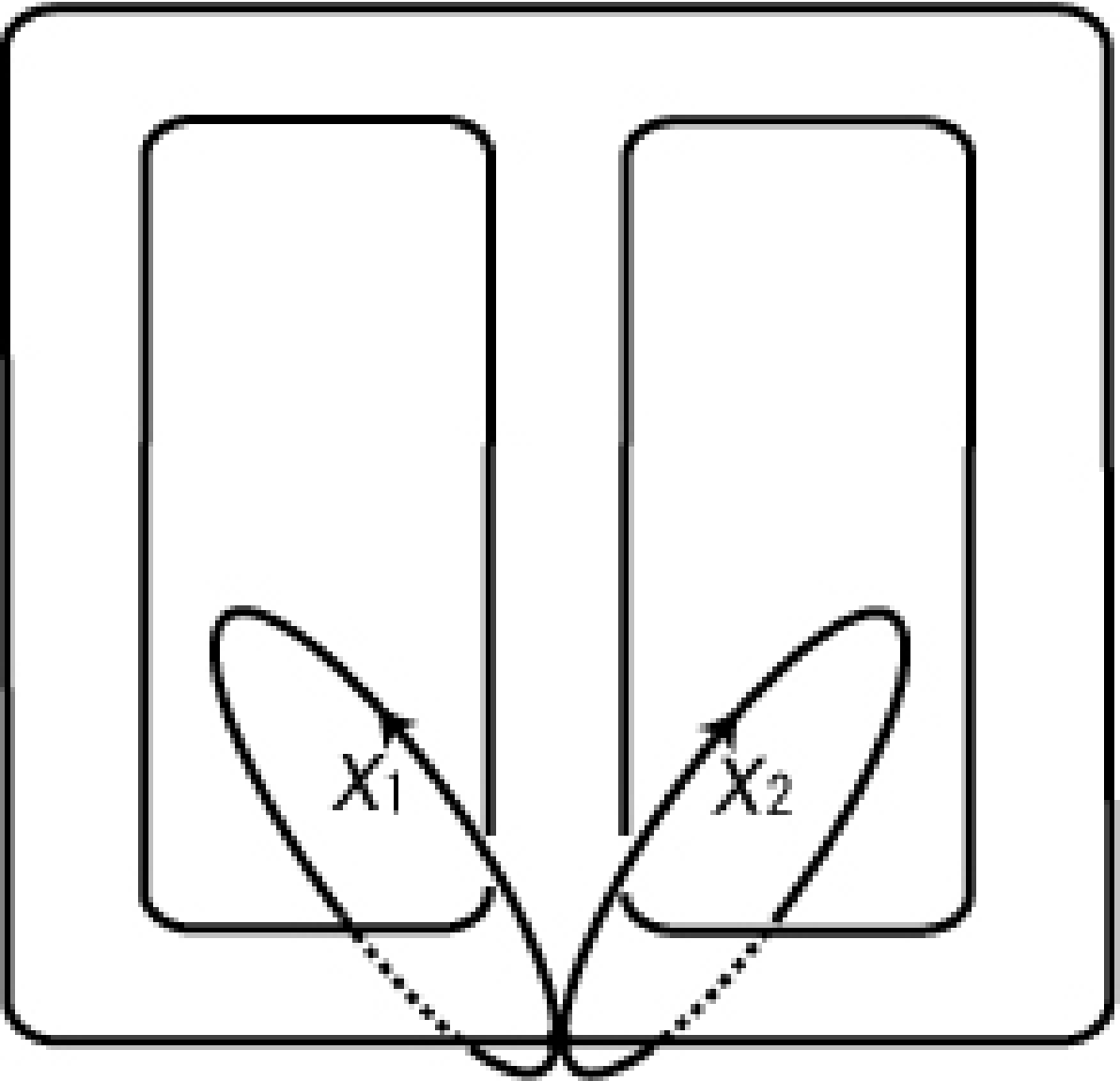}
\end{center}
\caption{A basis of $\pi_1 (M)$}
\label{fig:pretzel3}
\end{minipage}
\end{figure}

\noindent
The loops $x_1, x_2$ in Figure \ref{fig:pretzel3} form
a basis of $\pi_1 (M)$ of the complementary sutured manifold $M$.
They are oriented according to Figure \ref{fig:trefoil}.
A direct computation shows that 
\[i_-(\gamma_1)=x_1^{-1} (x_2 x_1)^2, \ \ 
i_-(\gamma_2)=x_2^{4}(x_2 x_1)^3, \ \ 
i_+(\gamma_1)=x_1^{-2} (x_1 x_2)^3, \ \ 
i_+(\gamma_2)=x_2^{5}(x_1 x_2)^2\]
and we obtain an admissible presentation 
\[{\small \left\langle 
\begin{array}{c|l}
i_-(\gamma_1),i_-(\gamma_2),
x_1,x_2,i_+(\gamma_1),i_+(\gamma_2) & 
\begin{array}{l}
i_-(\gamma_1)(x_1^{-1}x_2^{-1})^2x_1, \ 
i_-(\gamma_2)(x_1^{-1}x_2^{-1})^3x_2^{-4},\\
i_+(\gamma_1)(x_2^{-1}x_1^{-1})^3x_1^2, \ 
i_+(\gamma_2)(x_2^{-1}x_1^{-1})^2x_2^{-5} 
\end{array}
\end{array}\right\rangle}\]
of $\pi_1 (M)$. 
$H_1 (M)$ is the free abelian group generated by $t_1:=[x_1]$ and 
$t_2:=[x_2]$ and the natural homomorphism 
$\rho_\Gamma: \pi_1 (M) \to H_1 (M)=:\Gamma$ maps
\[i_-(\gamma_1)\mapsto t_1 t_2^2, \quad
i_-(\gamma_2)\mapsto t_1^3 t_2^{7}, \quad
i_+(\gamma_1)\mapsto t_1 t_2^{3}, \quad
i_+(\gamma_2)\mapsto t_1^2 t_2^{7}.\]

Now $\mathcal{K}_\Gamma$ is 
isomorphic to the field of rational functions 
with variables $t_1$ and $t_2$. 
We have
\[{}^{\rho_\Gamma}A=\begin{pmatrix}
I_2 & 0_{(2,2)}
\end{pmatrix}, \qquad 
{}^{\rho_\Gamma}B=\begin{pmatrix}
G_1 & G_2
\end{pmatrix}, \qquad 
{}^{\rho_\Gamma}C=\begin{pmatrix}
0_{(2,2)} & I_2 
\end{pmatrix},\]
where
\begin{align*}
G_1 &= \begin{pmatrix}
t_1 - t_1 t_2^{-1} -t_2^{-2} & 
-t_1^{-2}t_2^{-7}-t_1^{-1}t_2^{-6}-t_2^{-5} \\
-t_1-t_2^{-1} & 
-t_1^{-2}t_2^{-6}-t_1^{-1}t_2^{-5}-t_2^{-4}
-t_2^{-3}-t_2^{-2}-t_2^{-1}-1
\end{pmatrix},\\
G_2 &= \begin{pmatrix}
t_1-t_1 t_2^{-1}-t_2^{-2} & 
-t_1^{-1} t_2^{-6}-t_2^{-5} \\
-t_1^{-1}t_2^{-2} - t_1-t_2^{-1} & 
-t_1^{-2}t_2^{-6}-t_1^{-1}t_2^{-5}-t_2^{-4}
-t_2^{-3}-t_2^{-2}-t_2^{-1}-1
\end{pmatrix}.
\end{align*}
\noindent
Thus $\tau_{\Gamma}^+ (M) = 
\sideset{^{\rho_\Gamma}\!\!}{}
{\mathop{\begin{pmatrix} A \\ B \end{pmatrix}}\nolimits}
= \begin{pmatrix}
I_2 & 0_{(2,2)}\\
G_1 & G_2
\end{pmatrix}=G_2 
\in K_1 (\mathcal{K}_\Gamma)/\pm\rho_\Gamma (\pi_1 (M))$, 
which is non-trivial because 
\[\det (\tau_{\Gamma}^+ (M))= \det (G_2)
=
-t_1^{-1}t_2^{-6}-t_1+t_2^{-4}+t_2^{-3}+t_2^{-2}\]
is not a monomial. This shows that $P(-3,5,9)$ is {\it not} fibered
by Theorem \ref{thm:fiberness} (1). 
The Magnus matrix $r_\Gamma (M)$ is given by 

\[{\Large \left(\begin{array}{cc}
\frac{-1-t_1 t_2+t_1 t_2^2-t_1^2 t_2^4-t_1^2 t_2^5-t_1^2 t_2^6+t_1^3t_2^8}
{t_1 t_2^2(1-t_1 t_2^2 -t_1 t_2^3-t_1 t_2^4+t_1^2t_2^6)} & 
\frac{-1-t_1 t_2-t_1^2 t_2^2-t_1^2 t_2^3-t_1^2 t_2^4-t_1^2 t_2^5-t_1^2 t_2^6}
{t_1^3t_2^7(1-t_1 t_2^2-t_1 t_2^3-t_1 t_2^4+t_1^2 t_2^6)}\\ &\\
\frac{t_2^2( 1+t_1 t_2-t_1t_2^2)}
{1-t_1 t_2^2-t_1 t_2^3-t_1 t_2^4+t_1^2 t_2^6} & 
\frac{1+t_1 t_2+t_1^2 t_2^2+t_1^2 t_2^3-t_1^3 
t_2^5-t_1^3 t_2^6-t_1^3 t_2^7+t_1^4t_2^9}
{t_1^2t_2^3(1-t_1 t_2^2-t_1 t_2^3-t_1 t_2^4+t_1^2t_2^6)}
\end{array}\right),}\]

\noindent
which also indicates the non-fiberedness of $P(-3,5,9)$ since 
all the entries of $r_{\Gamma}(M)$ should be Laurent polynomials 
by Theorem \ref{thm:fiberness}(2) if it were fibered . 
\end{example}


\section{
Twisted homology and torsions of rationally homologically fibered 
link exteriors}\label{sec:factorization}

In this section, we see that the invariants defined 
in Section \ref{sec:magnus} make up 
torsions of exteriors of rationally homologically fibered links 
under special choices of PTFA groups $\Gamma$. Before that, we 
observe generalities of torsions of link exteriors. 

Let $L$ be an $n$-component link. 
Assume that 
the $($one variable\/$)$ 
Alexander polynomial $\Delta_L (t)$ of $L$ is not equal to zero. 
Then the Wirtinger presentation gives a 
presentation of $\pi_1 (E(L))$ with deficiency $0$. 
It is known that we can drop any one of the relations. 
Let $Q_0$ be such a presentation of the form
\[\langle y_1, \ldots, y_{m+1} \mid s_1, \ldots, s_m \rangle.\]
It is also known that the CW-complex $X(Q_0)$ constructed as in 
the proof of Proposition \ref{prop:Formula} has the same simple 
homotopy type as the link exterior $E(L)$. 

Let $\rho_\Gamma:\pi_1 (E(L)) \to \Gamma$ be an epimorphism 
whose target $\Gamma \neq \{1\}$ is PTFA and 
let $\rho: \pi_1 (E(L)) \to \langle t \rangle \cong \mathbb{Z}$ 
be the homomorphism 
sending each oriented meridian to $t$. The following proposition gives 
a sufficient condition for the torsion 
$\tau_\Gamma (E(L))=\tau (C_\ast (E(L);\mathcal{K}_\Gamma))$ 
of $E(L)$ to be defined.
\begin{proposition}
If the 
$($one variable\/$)$ 
Alexander polynomial $\Delta_L (t)$ of $L$ is not equal to zero and 
$\rho$ factors through $\rho_\Gamma$, then 
$H_\ast (E(L);\mathcal{K}_\Gamma)=0$.
\end{proposition}
\begin{proof}
The chain complex $C_\ast (X(Q_0); \mathbb{Z} \Gamma)$ is of the form
\begin{equation}\label{eq:chain}
0 \longrightarrow (\mathbb{Z} \Gamma)^m \xrightarrow{{}^{\rho_\Gamma}J \cdot} 
(\mathbb{Z} \Gamma)^{m+1} 
\xrightarrow{{}^{\rho_\Gamma} 
\begin{pmatrix}1-y_1^{-1}, 
\ldots, 1-y_{m+1}^{-1}\end{pmatrix}\cdot} 
\mathbb{Z} \Gamma 
\longrightarrow 0,
\end{equation}
\noindent
where $J=\overline{
\left(\displaystyle\frac{\partial s_j}{\partial y_i}
\right)}_{\begin{subarray}{c}
{}1 \le i \le m+1\\
1 \le j \le m
\end{subarray}}$. 
Now the assumption $\Delta_L (t) \neq 0$ 
implies that $H_\ast (E(L);\mathcal{K}_{\langle t \rangle})=0$. 
In particular, ${}^\rho J \cdot : (\mathbb{Z} \langle t \rangle)^m 
\to (\mathbb{Z} \langle t \rangle)^{m+1}$ is injective. 
Since PTFA groups are locally indicable, 
it follows from Friedl \cite[Proposition 6.4]{fri} that 
the second map of (\ref{eq:chain}) is injective. 
It is still injective when we apply $\otimes_\Gamma \mathcal{K}_\Gamma$. 
The third map of (\ref{eq:chain}) is clearly surjective after applying 
$\otimes_\Gamma \mathcal{K}_\Gamma$. Hence 
$H_\ast (E(L);\mathcal{K}_\Gamma) = H_\ast (X(Q_0);\mathcal{K}_\Gamma)
=0$ holds. 
\end{proof}
\begin{remark}
In the above argument, we can replace $\rho$ by any other homomorphism 
$\rho':\pi_1 (E(L)) \to \mathbb{Z}$ 
satisfying $H_\ast (E(L);\mathcal{K}_\mathbb{Z})=0$, 
where $\mathcal{K}_\mathbb{Z}$ is twisted by $\rho'$. 
In fact, since the multivariable Alexander polynomial of $L$ is non-trivial 
(see \cite[Proposition 7.3.10]{kook}, for example), we can 
use McMullen's argument \cite[Theorem 4.1]{mcm} to show that 
$H_\ast (E(L);\mathcal{K}_\mathbb{Z})=0$ for generic $\rho' \neq 0$. 
We also remark that by the definition of PTFA groups, 
there exists at least one homomorphism $\Gamma \to \mathbb{Z}$, 
whose composition with $\rho_\Gamma$ is non-trivial.
\end{remark}

Hereafter we assume that $H_\ast (E(L);\mathcal{K}_\Gamma)=0$. 
By using the cell structure of $X(Q_0)$, the torsion 
$\tau_\Gamma (E(L))$ is given by 
\begin{align*}
\tau_\Gamma (E(L))&=\tau_\Gamma (X(Q_0))= 
({}^{\rho_\Gamma} J)_i \cdot (1-\rho_\Gamma (y_i^{-1}))^{-1} \\
& \in 
K_1(\mathcal{K}_\Gamma)/\pm\rho_{\Gamma}(\pi_1 (E(L)))
=K_1(\mathcal{K}_\Gamma)/\pm \Gamma,
\end{align*}
\noindent
where $1 \le i \le m+1$ is chosen so that $\rho_\Gamma (y_i) \neq 1$ and 
$({}^{\rho_\Gamma} J)_i$ is obtained from ${}^{\rho_\Gamma} J$ by deleting 
its $i$-th row (see Friedl \cite[Lemma 6.6]{fri} for example). 
The torsion $\tau_\Gamma (E(L))$ is independent of 
such a choice of $i$. 

For later use, we show that we can compute $\tau_\Gamma (E(L))$ from 
any presentation of $\pi_1 (E(L))$ with deficiency $1$. 
Suppose $Q$ is such a presentation of the form 
\[\langle x_1, \ldots, x_{k+1} \mid r_1, \ldots, r_k \rangle.\]
Let $X(Q)$ be the corresponding 2-complex. 

\begin{lemma}\label{lem:simple}
The equality 
$\tau_\Gamma (E(L))=\tau_\Gamma (X(Q)) \in 
K_1(\mathcal{K}_\Gamma)/\pm \Gamma$ holds. 
\end{lemma}
\begin{proof} 
The existence of the presentation $Q_0$ 
shows that the deficiency of $\pi_1 (E(L))$ is at least $1$. 
On the other hand, if it were greater than $1$, then 
$H_1 (E(L);\mathcal{K}_\Gamma)$ should be non-trivial, 
a contradiction. 
Therefore the deficiency of $\pi_1 (E(L))$ is $1$. Then 
Harlander-Jensen's theorem \cite[Theorem 3]{hj} shows that 
$X (Q_0)$ and $X (Q)$ are homotopy equivalent. 
In fact they are simple homotopy equivalent by 
Waldhausen's theorem \cite[Theorems 19.4, 19.5]{waldhausen1}. 
Hence 
\[\tau_\Gamma (E(L))=\tau_\Gamma (X(Q_0))=\tau_\Gamma (X(Q))\] 
holds. 
\end{proof}

Now we assume that $L$ is an $n$-component 
rationally homologically fibered link 
with a minimal genus Seifert surface $R$ of genus $g$. 
Let $M=(M,i_+,i_-) \in \mathcal{C}_{g,n}^\mathbb{Q}$ 
be a rational homology cylinder over $\Sigma_{g,n}$ 
obtained as the complementary sutured manifold for $R$. 
We take a basepoint $p$ of $M$ on a component of 
$i_{+}(\partial \Sigma_{g,n})=i_{-}(\partial \Sigma_{g,n})$ and 
a small segment $\mu_0 \subset \partial M$ 
which intersects with $i_\pm(\partial \Sigma_{g,n})$ 
at $p$ transversely. $\mu_0$ is oriented so that it goes across 
$i_\pm(\partial \Sigma_{g,n})$ from 
$i_+ (\Sigma_{g,n})$ to $i_-(\Sigma_{g,n})$. 
We may assume that $\mu_0$ defines a meridian loop 
$\mu \in \pi_1 (E(L))$ when we remake $E(L)$ from $M$. 
By the definition of a PTFA group, any meridian loop of at least 
one component of $L$ must satisfies 
$\rho_\Gamma (\mu) \neq 1 \in \Gamma$, and we choose such a $\mu$ 
by changing the basepoint if necessary. 

Consider the composition 
$\pi_1 (M) \to \pi_1 (E(L)) 
\xrightarrow{\rho_\Gamma} \Gamma$ to 
define $r_\Gamma (M)$ and $\tau_\Gamma^+ (M)$. 

\begin{theorem}\label{thm:factorization}
Under the above assumptions, we have 
\begin{align*}
\tau_{\Gamma}(E(L)) &=
\tau_{\Gamma}^+ (M) \cdot (I_{2g+n-1}- 
\rho_{\Gamma}(\mu) r_{\Gamma}(M)) 
\cdot (1-\rho_{\Gamma}(\mu))^{-1} \\
& \in K_1(\mathcal{K}_\Gamma)/\pm \Gamma.
\end{align*}
\end{theorem}
\begin{proof}
Given an admissible presentation of $\pi_1 (M)$ 
as in (\ref{admissible}), 
we denote it briefly by 
\[\pi_1 (M) \cong \langle i_- (\overrightarrow{\gamma}), 
\overrightarrow{z}, i_+ (\overrightarrow{\gamma}) \mid 
\overrightarrow{r} \rangle.\]
>From this, we can obtain a presentation $Q_1$ of 
$\pi_1 (E(L))$ given by 
\[\pi_1 (E(L)) \cong \langle i_- (\overrightarrow{\gamma}),
\overrightarrow{z},i_+ (\overrightarrow{\gamma}),\mu \mid 
\overrightarrow{r}, i_- (\overrightarrow{\gamma}) \, \mu \, 
i_+(\overrightarrow{\gamma})^{-1} \mu^{-1} \rangle.\]
Consider the 2-complex $X(Q_1)$ as before. 
The matrix 
\[J:=
\begin{pmatrix}
A & I_{2g+n-1} \\
B & 0_{(l,2g+n-1)} \\
C & -\rho_\Gamma (\mu)^{-1} I_{2g+n-1} \\
0_{(1,2g+n-1+l)} & \ast \ \ast \ \cdots \ \ast
\end{pmatrix}\]
represents the boundary map 
\[C_2(X(Q_1);\mathcal{K}_\Gamma) \cong \mathcal{K}_\Gamma^{4g+2n-2+l} 
\longrightarrow \mathcal{K}_\Gamma^{4g+2n-1+l} 
\cong C_1 (X(Q_1);\mathcal{K}_\Gamma),\]
where we use 
the above admissible presentation of $\pi_1 (M)$ to 
give the matrices $A$, $B$ and $C$ 
(recall Section \ref{sec:magnus}), and then 
apply $\rho_\Gamma$ to their entries for simplicity. 
 
By Lemma \ref{lem:simple}, we have 
\[\tau_\Gamma (E(L)) = \tau_\Gamma (X(Q_1)) 
= J_\mu \cdot (1-\rho_\Gamma(\mu)^{-1})^{-1},\]
where $J_\mu$ is obtained from $J$ 
by deleting the last row, and 
Then as elements 
in $K_1(\mathcal{K}_\Gamma)/\pm\rho_{\Gamma}(\pi_1 (E(L)))$, we have 

\begin{align*}
J_\mu &=
\begin{pmatrix}
A & I_{2g+n-1} \\
B & 0_{(l,2g+n-1)}\\
C & -\rho_\Gamma (\mu)^{-1} I_{2g+n-1}
\end{pmatrix}
=\begin{pmatrix}
A + \rho_\Gamma (\mu) C & 0_{2g+n-1} \\
B & 0_{(l,2g+n-1)}\\
C & -\rho_\Gamma (\mu)^{-1} I_{2g+n-1}
\end{pmatrix}\\
&=\begin{pmatrix}
A + \rho_\Gamma (\mu) C \\ B 
\end{pmatrix}
=\begin{pmatrix}
A \\ B
\end{pmatrix}
-\rho_\Gamma (\mu)
\begin{pmatrix}
r_{\Gamma} (M) \quad Z \\
0_{(l,2g+n-1+l)}
\end{pmatrix}
\begin{pmatrix}
A \\ B
\end{pmatrix}\\
&=\begin{pmatrix}
I_{2g+n-1} - \rho_\Gamma (\mu) r_{\Gamma} (M) & 
-\rho_\Gamma (\mu) Z \\
0_{(l,2g+n-1)} & I_l
\end{pmatrix}
\begin{pmatrix}
A \\ B
\end{pmatrix}\\
&=(I_{2g+n-1} - \rho_\Gamma (\mu) r_{\Gamma} (M)) 
\begin{pmatrix}
A \\ B 
\end{pmatrix}=
(I_{2g+n-1} - \rho_\Gamma (\mu) r_{\Gamma} (M)) \cdot 
\tau_\Gamma^+ (M),
\end{align*}
\noindent
where $Z$ is defined by the formula 
$(r_\Gamma(M) \quad Z)= 
-C \begin{pmatrix} A \\ B \end{pmatrix}^{-1}$ 
(see Proposition \ref{prop:Formula} (2)). 
This completes the proof. 
\end{proof}

\begin{example}
(1) Consider the homomorphism 
$\rho:\pi_1 (E(L)) \to \langle t \rangle$ 
at the beginning of this section. We have 
$t=\rho (\mu)$. 
It is easy to see that 
the composition $H_1 (M) \to H_1 (E(L)) \xrightarrow{\rho} 
\mathbb{Z}$ is trivial. 
Thus the matrices $\tau_{\langle t \rangle}^+ (M)$ and 
$r_{\langle t \rangle} (M)$ 
have their entries in $\mathbb{Q}$ and in fact 
$r_{\langle t \rangle} (M)=\sigma (M)$ holds. 
Then Theorem \ref{thm:factorization} together with 
Milnor's formula \cite[Section 2]{milnor2} give a factorization 
\begin{align*}\Delta_L (t) 
&= (1-t)\det (\tau_{\langle t \rangle} (E(L))) \\
&= \det (\tau_{\langle t \rangle}^+ (M)) \cdot \det 
\big(I_{2g+n-1} - t\sigma (M) \big)
\end{align*}
\noindent
of the (one variable) Alexander polynomial of $L$. 
This formula is essentially the same as 
(\ref{eq:factor}) 
in the proof of Proposition \ref{thm:homologicalfibre}. 

\noindent
(2) Let $\pi_1 (E(L)) \to H:= H_1 (E(L)) \cong \mathbb{Z}^n$ 
be the abelianization homomorphism for $n \ge 2$. 
In this case, Theorem \ref{thm:factorization} together with 
Milnor's formula give a factorization 
\begin{align*}
\Delta(L) &= \det (\tau_H (E(L))) \\
&=\frac{1}{1-\rho_H (\mu)} \cdot 
\det (\tau_H^+ (M)) \cdot 
\det \big( I_{2g+n-1} - \rho_\Gamma (\mu)r_H (M) \big) 
\end{align*}
\noindent
of the multivariable Alexander polynomial 
$\Delta (L)$ of $L$. 
\end{example}

More examples are given in 
\cite{gs10}, where we 
detect the non-fiberedness of the thirteen knots mentioned in 
Example \ref{ex:FK} by using the torsions associated with 
the metabelian quotients of their knot groups. 


\section{The handle number}\label{section:handle}

In this section, we review the handle number of 
a sutured manifold according to \cite{goda1,goda2}. 

A {\it compression body\/} is a cobordism $W$ relative to 
the boundary 
between surfaces $\partial_{+}W$ and $\partial_{-}W$ 
such that $W$ is diffeomorphic to 
$\partial_{+}W\times [0,1] \cup \text{ (2-handles) } 
\cup \text{ (3-handles)}$ and $\partial_{-}W$ has no 
2-sphere components. 
In this paper, we assume $W$ is connected.
If $\partial_{-}W=\emptyset$, $W$ is a handlebody. 
If $\partial_{-}W\neq\emptyset$, 
$W$ is obtained from $\partial_{-}W\times [0,1]$ 
by attaching a number of 1-handles along the disks 
on $\partial_{-}W\times\{ 1 \}$ 
where $\partial_{-}W$ corresponds to 
$\partial_{-}W\times\{ 0\}.$ 
We denote by $h(W)$ the number of these attaching 
1-handles.

Let $(M,\gamma)$ be a sutured manifold such that 
$R_{+}(\gamma)\cup R_{-}(\gamma)$ has no 2-sphere 
components. 
We say that $(W,W')$ is a {\it Heegaard splitting\/} 
of $(M,\gamma)$ if both $W$ and $W'$ are compression bodies, 
$M=W\cup W'$ with 
$W\cap W'=\partial_{+}W=\partial_{+}W',\, 
\partial_{-}W=R_{+}(\gamma),$ and $\partial_{-}W'=R_{-}(\gamma)$.

\begin{definition}\label{def:handle}
Assume that $R_{+}(\gamma)$ is diffeomorphic to $R_{-}(\gamma)$. 
We define the {\it handle number\/} of $(M,\gamma)$ as follows: 
\[h(M,\gamma)=\min\{h(W)(=h(W'))~|~(W,W') 
\text{ is a Heegaard splitting of }(M,\gamma)\}.\]
If $(M,\gamma)$ is the complementary sutured manifold 
for a Seifert surface $R$, we define 
\[h(R)=\min\{h(W)~|~(W,W') \text{ is a Heegaard splitting of }(M,\gamma)\},\] 
and call it the handle number of $R$.
\end{definition} 

If $(M,\gamma)$ is a product sutured manifold
then $h(M,\gamma)=0$, and vice versa.
For the behavior and some estimates of the handle number, 
see \cite{goda3,goda4}. 
Note that this invariant is closely related to the Morse-Novikov number 
for knots and links \cite{vpr}. 

Here we present an estimate of the handle number using the homology. 
For a sutured manifold $(M,\gamma)$, 
fix two diffeomorphisms 
$i_{\pm}:\Sigma_{g,n} \stackrel{\cong}{\to} R_{\pm}(\gamma)$ 
as in the previous sections. 
Suppose $M$ has a Heegaard splitting $(W,W')$ such that $h(W)=h$. 
Then, $M$ is diffeomorphic to a manifold obtained from 
$R_{+}(\gamma)\times [0,1]$ by attaching $h$ 1-handles 
and $h$ 2-handles. By considering the computation of 
$H_{1}(M,i_{+}(\Sigma_{g,n}))$ from this handle decomposition, 
we have 
\[h(M,\gamma)\ge p,\]
where $p$ is the minimum number of generators of 
$H_{1}(M, i_{+}(\Sigma_{g,n}))$. 
This estimate is effective in general (see \cite[Example 6.3]{goda2}), 
however not at all in case $(M,\gamma)$ is a homology cylinder. 
To obtain a method which works in that case, 
we consider a local coefficient system $\mathcal R$ of a ring on $M$. 
By the same argument as above, we have:

\begin{proposition}\label{prop:estimate}
$h(M,\gamma)$ is greater than or equal to 
the minimum number of elements generating 
$H_{1}(M,i_{+}(\Sigma_{g,n});\mathcal R)$ 
as an $\mathcal R$-module.
\end{proposition}


\section{A lower estimate of handle numbers of 
doubled knots by using Nakanishi index}\label{sec:nakanishi}

In this section, we give a lower estimate of handle numbers 
of genus one Seifert surfaces for doubled knots (\cite[page 20]{bz}) by 
using a machinery similar to the $\Gamma$-torsion. 

Let $\widetilde{K}$ be the knot in $S^1 \times D^2$ depicted 
in Figure \ref{fig:double1}, where 
$\widetilde{V}=S^1 \times D^2$ is supposed to be embedded in $S^3$ 
in a standard position. 
We denote by $\widetilde{\lambda}$ 
the standard longitude of $S^1 \times D^2$. 
Take a Seifert surface $\widetilde{R}$ of $\widetilde{K}$ 
as in the figure.

\begin{figure}[h]
\centering
\includegraphics[width=.45\textwidth]{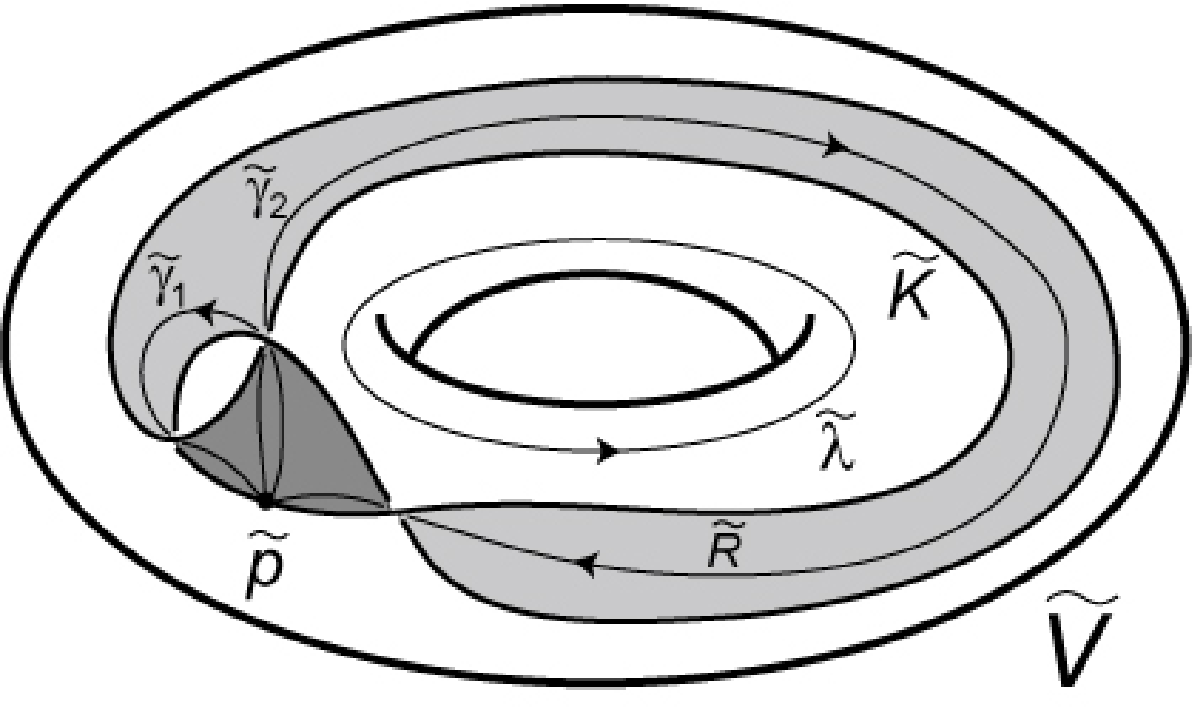}
\caption{The knot $\widetilde{K}$ in $S^1 \times D^2$}\label{fig:double1}
\end{figure}

For a knot $\widehat{K}$ 
(not necessarily homologically fibered) in $S^3$, we 
take a tubular neighborhood $N(\widehat{K})$ of $\widehat{K}$. 
Attaching $\widetilde{V}$ to ${\rm cl}(S^3-N(\widehat{K}))$, 
we obtain a doubled knot $K$ in $S^3$ 
with the Seifert surface $R$. 

If we attach $\widetilde{V}$ to ${\rm cl}(S^3-N(\widehat{K}))$
by gluing $\widetilde{\lambda}$ to the 0-framing of 
$\partial N(\widehat{K})$, 
then we have the Seifert surface $R$ whose Seifert matrix is the same as 
that of $\widetilde{R}$. Therefore, as seen in Example \ref{ex:gkm}, 
if $\widehat{K}$ is homologically fibered, 
so is $K$. 

\begin{proposition}\label{prop:nakanishi}
The handle number $h(R)$ of $R$ is greater than or equal to 
the Nakanishi index $m(\widehat{K})$ of $\widehat{K}$. 
\end{proposition}

\noindent
Recall that the {\it Nakanishi index} $m(\widehat{K})$ of 
a knot $\widehat{K}$ is the minimum size of 
square matrices representing $H_1 (G_{\widehat{K}};\mathbb{Z} [t^\pm])$ as 
a $\mathbb{Z} [t^\pm]$-module, where $G_{\widehat{K}}$ is the knot group of 
$\widehat{K}$ and 
$t$ is a generator of the abelianization of $G_{\widehat{K}}$. 
($H_1 (G_{\widehat{K}};\mathbb{Z} [t^\pm])$ is 
nothing other than the first homology group 
of the infinite cyclic cover of the knot exterior of $\widehat{K}$.) 
It is shown in Kawauchi \cite{kawauchi2} that 
\[m(\widehat{K}) =e(H_1 (G_{\widehat{K}};\mathbb{Z} [t^\pm])),\] 
where $e(A)$ of a $\mathbb{Z} [t^\pm]$-module $A$ 
is the minimal number of elements generating 
$A$ over $\mathbb{Z} [t^\pm]$.

\begin{proof}[Proof of Proposition $\ref{prop:nakanishi}$]
Since $h(R) \ge e(H_1 (M,i_+(\Sigma_{1,1});\mathbb{Z} [t^\pm]))$ 
by Proposition \ref{prop:estimate}, 
it suffices to show that 
$e(H_1 (M,i_+(\Sigma_{1,1});\mathbb{Z} [t^\pm])) \ge m(\widehat{K})$.

Let $\langle \widetilde{\gamma}_1, \widetilde{\gamma}_2 \rangle$ 
be a generating system of 
$\pi_1 (\widetilde{R},\widetilde{p})$ as in Figure \ref{fig:double1}.
We denote by $\gamma_{i}$ $(i=1,2)$ 
the image of $\tilde{\gamma}_{i}$ in $R$ and 
denote by $p$ the image of $\widetilde{p}$.
Further, we denote by $(M,\gamma)$ the complementary sutured manifold for $R$.
It is easy to see that 
a presentation of $\pi_1 (M,p)$ can be obtained by adding a generator 
$x$ to the Wirtinger presentation 
$\langle x_1, x_2, \ldots, x_l \mid r_1, \ldots, r_{l-1} \rangle$ 
of $G_{\widehat{K}}$ (with basepoint $p$) as shown in Figure \ref{fig:double2}.

\begin{figure}[h]
\centering
\includegraphics[width=.85\textwidth]{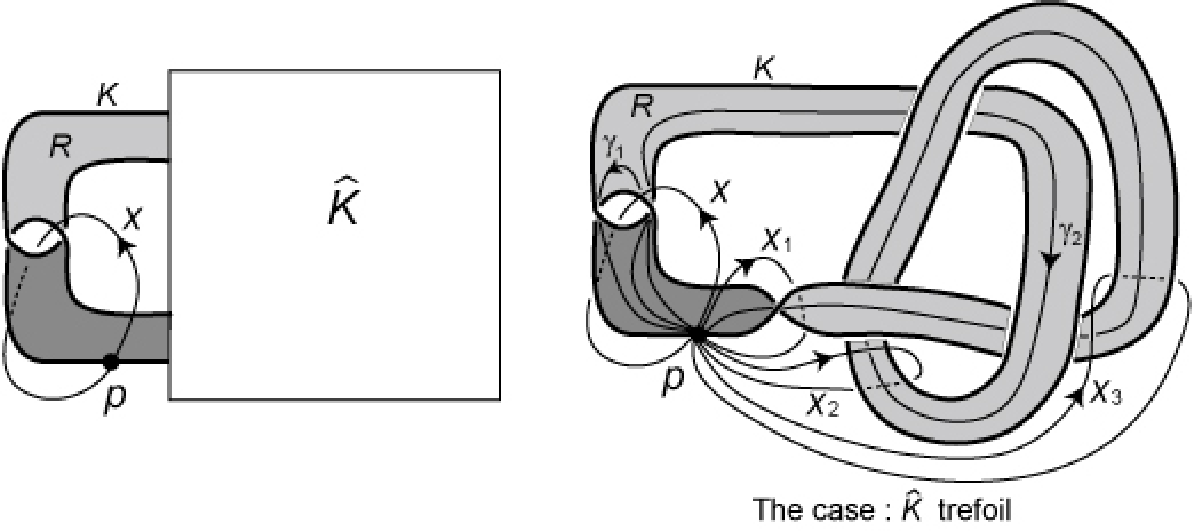}
\caption{Doubled knot}\label{fig:double2}
\end{figure}

>From these data, 
we can give an admissible presentation of $\pi_1 (M,p)$ as follows:
\[
\pi_{1}(M,p)\cong 
\left\langle
\begin{array}{c|c}
\begin{array}{l}
i_- (\gamma_1), i_-(\gamma_2),\\
x, x_1, x_2, \ldots, x_l,\\
i_+ (\gamma_1), i_+ (\gamma_2)
\end{array}&
\begin{array}{l}
i_-(\gamma_1)w_1 x, i_- (\gamma_2) w_2,\\
r_1, \ldots , r_{l-1},\\ 
i_+ (\gamma_1) x, i_+(\gamma_2)x w_3
\end{array}
\end{array}\right\rangle,\]
where $w_1, w_2, w_3$ are words in $x_1,\ldots,x_l$. 
The abelianization map 
$\rho: \pi_1 (M) \to H_1 (M)\cong \mathbb{Z}^2 = 
\mathbb{Z}s \oplus \mathbb{Z}t$ is 
given by 
\[x \mapsto s,\qquad x_1,x_2,\ldots,x_l \mapsto t.\]

A computation in matrices with entries in $\mathbb{Z} H_1 (M)
=\mathbb{Z}[s^\pm,t^\pm]$ shows that 
\[\sideset{^{\rho}\!\!}{}
{\mathop{\begin{pmatrix} A \\ B \\ C \end{pmatrix}}\nolimits}
=
\bordermatrix{
&i_-(\gamma_1)w_1 x & i_- (\gamma_2) w_2 & 
r_1 & \ldots & r_{l-1} & i_+ (\gamma_1) x & 
i_+(\gamma_2)x w_3 \cr 
i_-(\gamma_1)&1&0&0&\cdots&0&0&0\cr
i_-(\gamma_2)&0&1&0&\cdots&0&0&0\cr
x&\ast&0&0&\cdots&0&s&\ast\cr
x_1&\ast&\ast&a_{11}&\cdots&a_{1, l-1}&0&b_1\cr
\vdots&\vdots&\vdots&\vdots&\ddots&\vdots&\vdots&\vdots&\cr
x_l&\ast&\ast&a_{l1}&\cdots&a_{l, l-1}&0&b_l\cr
i_+(\gamma_1)&0&0&0&\cdots&0&1&0\cr
i_+(\gamma_2)&0&0&0&\cdots&0&0&1
},\]
where 
$a_{ij}=\overline{\displaystyle\frac{\partial r_j}{\partial x_i}}$ \ 
coincides with the $(i,j)$-entry (applied an involution) of the 
Alexander matrix with respect to 
the Wirtinger presentation of $G_{\widehat{K}}$, and 
$b_i= \overline{\displaystyle\frac{\partial(i_+(\gamma_2)x w_3)}
{\partial x_i}}$.

Recall that the matrix 
$\sideset{^{\rho}\!\!}{}
{\mathop{\begin{pmatrix} A \\ B \end{pmatrix}}\nolimits}$ 
gives a representation matrix of 
$H_1 (M,i_+(\Sigma_{1,1});\mathbb{Z}[s^\pm,t^\pm])$. 
As a representation matrix, 
$\sideset{^{\rho}\!\!}{}
{\mathop{\begin{pmatrix} A \\ B \end{pmatrix}}\nolimits}$ 
is equivalent to 
\[\begin{pmatrix}
a_{11}&\cdots&a_{1, l-1}&b_1\\
\vdots&\ddots&\vdots&\vdots\\
a_{l1}&\cdots&a_{l, l-1}&b_l
\end{pmatrix}.\]
Therefore, if we apply the natural map 
$\mathbb{Z} [s^\pm, t^\pm] \to 
\mathbb{Z} [t^\pm]$ $(s \mapsto 1)$ to each entry, 
we have an exact sequence
\[\mathbb{Z}[t^\pm] \longrightarrow 
H_1 (G_{\widehat{K}},\{1\};\mathbb{Z}[t^\pm]) 
\longrightarrow H_1 (M,i_+(\Sigma_{1,1});\mathbb{Z} [t^\pm]) 
\longrightarrow 0,\] 
which shows that 
\begin{align}\label{ineq1}
e(H_1 (G_{\widehat{K}},\{1\};\mathbb{Z}[t^\pm])) 
\le e(H_1 (M,i_+(\Sigma_{1,1});\mathbb{Z} [t^\pm])) +1.
\end{align}
\noindent
(Recall that the Alexander matrix of $\widehat{K}$ is a presentation matrix of 
$H_1 (G_{\widehat{K}},\{1\};\mathbb{Z}[t^\pm])$.)

In the homology exact sequence
\[0 \longrightarrow H_1 (G_{\widehat{K}};\mathbb{Z}[t^\pm]) 
\longrightarrow H_1 (G_{\widehat{K}},\{1\};\mathbb{Z}[t^\pm]) \longrightarrow 
H_0 (\{1\};\mathbb{Z}[t^\pm]) \longrightarrow 
H_0(G_{\widehat{K}};\mathbb{Z}[t^\pm]),\]
the fourth map is given by the augmentation map
\[H_0 (\{1\};\mathbb{Z}[t^\pm]) \cong \mathbb{Z}[t^\pm] 
\longrightarrow \mathbb{Z} \cong H_0(G_{\widehat{K}};\mathbb{Z}[t^\pm]), 
\qquad (t \mapsto 1),\]
whose kernel is $(t-1) \mathbb{Z}[t^\pm] \cong \mathbb{Z}[t^\pm]$, 
a free $\mathbb{Z}[t^\pm]$-module. 
Hence, we obtain an exact sequence
\[0 \longrightarrow H_1 (G_{\widehat{K}};\mathbb{Z}[t^\pm]) 
\longrightarrow H_1 (G_{\widehat{K}},\{1\};\mathbb{Z}[t^\pm]) \longrightarrow 
\mathbb{Z}[t^\pm] \longrightarrow 0.\]
Then, by \cite[Lemma 2.5]{kawauchi2}, we have 
\begin{equation}\label{ineq2}
e(H_1 (G_{\widehat{K}},\{1\};\mathbb{Z}[t^\pm])) 
= e(H_1 (G_{\widehat{K}};\mathbb{Z}[t^\pm])) +1 = m(\widehat{K}) +1.
\end{equation}
\noindent
The conclusion follows from (\ref{ineq1}) and (\ref{ineq2}). 
\end{proof}

\begin{corollary}
There exist homologically fibered knots 
having Seifert surfaces of genus 1 with arbitrarily large handle number.
\end{corollary}
\begin{proof}
It is known that there exist knots with arbitrarily large 
Nakanishi index. 
Our claim follows by combining this fact with 
Proposition \ref{prop:nakanishi}. 
\end{proof}

\begin{example}
We present an example which shows the estimate of
Proposition \ref{prop:nakanishi} is sharp.

Let $\widehat{K}$ be the pretzel knot $P(3,-3,3)=9_{46}$.
The Nakanishi index of $\widehat{K}$ is 2 from the list in \cite{kook}.  
Let $K$ be a doubled knot along $\widehat{K}$ 
and let $\tau_{1}$ and $\tau_{2}$ 
(resp. $\tau'_{1}$ and $\tau'_{2}$) be 
the arcs whose ends are in $+$-side (resp. $-$-side) 
of the Seifert surface $R$ as illustrated 
in Figure \ref{fig:double10}.

\begin{figure}[h]
\centering
\includegraphics[width=.8\textwidth]{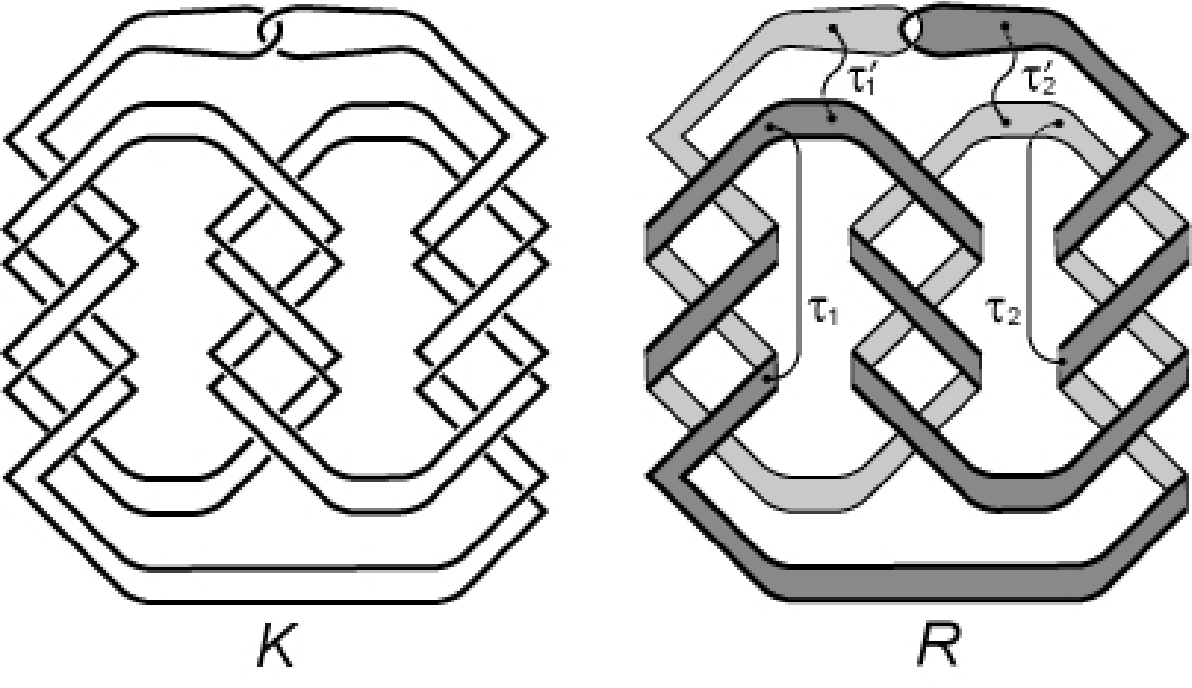}
\caption{Doubled knot $K$ obtained from $P(3,-3,3)$}\label{fig:double10}
\end{figure}

Let $(M,\gamma)$ be the complementary sutured manifold for $R$. 
Then we can observe that $({\rm cl}(M-N(\tau_{1}\cup\tau_{2} 
\cup\tau'_{1}\cup\tau'_{2})),\gamma)$, say $(\check{M},\gamma)$,  
is also a sutured manifold.
Furthermore, we can show that $(\check{M},\gamma)$ is 
a product sutured manifold 
by using the technique of product decompositions (see Gabai \cite{gabai2}).
This means that 
$(M,\gamma)$ has a Heegaard splitting $(W,W')$ 
such that  $h(W)=h(W')=2$ where 
$\tau_{1}$ and $\tau_{2}$ (resp. $\tau'_{1}$ and $\tau'_{2}$) 
correspond to 
the attaching 1-handles of $W$ (resp. $W'$).
Thus we have $h(R)\le 2$. 
(See \cite{goda4} for the detail of this technique.)
Therefore we have $h(R)=2$ by Proposition \ref{prop:nakanishi}.
Note that the Alexander polynomial of $K$ is equal to 
$t^2-t+1$, namely $K$ is homologically fibered. 
\end{example}


\bibliographystyle{amsplain}

\end{document}